\documentclass[10pt,centertags]{amsart}
\usepackage{amsfonts}
\usepackage{}
\usepackage{amsmath,amstext,amsthm,amssymb,amscd}
\usepackage[mathscr]{eucal}
\usepackage{mathrsfs}
\usepackage{epsf}
\usepackage{tikz}

\usetikzlibrary{matrix,arrows,decorations.pathmorphing}

\numberwithin{equation}{section}

\newcommand{\field}[1]{\mathbb{#1}}
\newcommand{\Z}{\field{Z}}
\newcommand{\R}{\field{R}}
\newcommand{\C}{\field{C}}

\def\cC{\mathscr{C}}

\def\cE{\mathscr{E}}

\def\mA{\mathcal{A}}
\def\mB{\mathcal{B}}
\def\mD{\mathcal{D}}
\def\mE{\mathcal{E}}
\def\mF{\mathcal{F}}

\def\mT{\mathcal{T}}
\def\mU{\mathcal{U}}
\def\mV{\mathcal{V}}

\newcommand\mS{\mathcal{S}}

\def\Im{{\rm Im}}

\def\la{\langle}
\def\ra{\rangle}

\DeclareMathOperator{\End}{End}

\DeclareMathOperator{\Id}{Id}

\DeclareMathOperator{\tr}{Tr}

\DeclareMathOperator{\ind}{ind}

\DeclareMathOperator{\ch}{ch}

\newtheorem{thm}{Theorem}[section]
\newtheorem{lemma}[thm]{Lemma}
\newtheorem{prop}[thm]{Proposition}
\newtheorem{cor}[thm]{Corollary}
\theoremstyle{definition}
\newtheorem{rem}[thm]{Remark}
\theoremstyle{definition}
\newtheorem{defn}[thm]{Definition}

\newcommand{\be}{\begin{eqnarray}}
\newcommand{\ee}{\end{eqnarray}}
\newcommand{\ov}{\overline}

\newcommand{\var}{\varepsilon}
\numberwithin{equation}{section}
\numberwithin{thm}{section}

\newcommand{\comment}[1]{}

\begin{document}
	
	\title{Real embedding and equivariant eta forms}
	
	\author{Bo Liu}
	
	\address{Department of Mathematics, 
		East China Normal University, 
		500 Dongchuan Road, 
		Shanghai, 200241 
		P.R. China}
	\email{boliumath@outlook.com}
	
	\begin{abstract}
	In \cite{MR1214954}, Bismut and Zhang establish a $\mathrm{mod}\, \Z$ 
	embedding formula of Atiyah-Patodi-Singer reduced eta 
	invariants. 
	In this paper, we explain the hidden $\mathrm{mod}\, \Z$ term as a 
	spectral flow
	and 
	extend this embedding formula to the 
	equivariant family case. In this case, the spectral flow is generalized 
	to  
	the equivariant chern character of some equivariant Dai-Zhang higher 
	spectral flow.
	\end{abstract}
	\maketitle

	
	
	\tableofcontents
	
	\setcounter{section}{-1}
	
	\section{Introduction} \label{s00}
	
	Let $i:Y\rightarrow X$ be an embedding between two odd dimensional 
	closed
	oriented spin manifolds. For any Hermitian vector bundle 
	$\mu$ over 
	$Y$
	carrying a Hermitian connection, 
	under a natural assumption,
	Bismut and Zhang \cite{MR1214954} 
	establish
	a $\mathrm{mod}\, \Z$ formula, expressing the Atiyah-Patodi-Singer 
	reduced eta invariant \cite{MR0397797} of 
	certain
	direct image of $\mu$ over $X$,
	through 
	the reduced eta invariant of the bundle $\mu$ over $Y$,  up to some 
	geometric
	Chern-Simons current. 
	
	In this paper, we explain the hidden $\mathrm{mod}\, \Z$ term as a 
	spectral flow in 
	Bismut-Zhang embedding formula and 
	 extend this embedding formula to the 
	equivariant family case. In this case, the spectral flow is generalized 
	to  
	the equivariant chern character of some equivariant Dai-Zhang higher 
	spectral flow \cite{MR1638328}.
	
	The main motivation of this generalization is to look for a general
	Grothendieck-Riemann-Roch theorem in the equivariant differential 
	K-theory, which is already established in many important cases 
	\cite{MR2602854,Bunke2009,MR3182258}. Roughly speaking, the differential 
	K-theory is the smooth 
	version of the arithmetic K-theory in Arakelov geometry. Our main result 
	here is expected to play the same role in the equivariant differential 
	K-theory of the Bunke-Schick model \cite{Bunke2009,MR3182258,Liu2016} as 
	the Bismut-Lebeau embedding formula \cite{MR1188532} does in the proof 
	of Arithmetic 
	Grothendieck-Riemann-Roch theorem in Arakelov 
	geometry.

	In this paper, we use the language of the Clifford modules.

Let $\pi:W\rightarrow B$ be a smooth submersion of smooth oriented 
manifolds with closed fibres $Y$. Let $TY=TW/B$ be the relative tangent 
bundle to the fibres $Y$. Let $T_{\pi}^HW$ be a horizontal subbundle of 
$TW$. Let $g^{TY}$ be a Riemannian metric on $TY$.	Let $C(TY)$ be the 
Clifford algebra bundle of $(TY, g^{TY})$ and $(\mE,h^{\mE})$ be
a $\Z_2$-graded self-adjoint $C(TY)$-module with Clifford connection 
$\nabla^{\mE}$ (cf. (\ref{bl0959}) and (\ref{e01024})). Let $G$ be a compact 
Lie group 
which acts on $W$ and $B$ such that for any 
$g\in G$, $\pi\circ g=g\circ\pi$. We assume that the action of $G$ preserves 
the horizontal bundle and 
the orientation of $TY$ and could be lifted on $\mE$ such that it is 
compatible with the Clifford action and the $\Z_2$-grading.
We assume that $g^{TY}$, $h^{\mE}$, $\nabla^{\mE}$ are $G$-invariant. 
For any $g\in G$, if the group action is trivial on $B$, the equivariant 
Bismut-Cheeger eta form 
$\tilde{\eta}_g(\mF, \mA)\in 
\Omega^*(B,\C)/\mathrm{Im}\,d$ is defined in Definition \ref{d017} up to 
exact forms with respect to the equivariant geometric family $\mF=(W, \mE, 
T_{\pi}^HW, g^{TY}, h^{\mE}, \nabla^{\mE})$ (cf. Definition \ref{d001}) over 
$B$ 
and a
perturbation operator $\mA$ (cf. Definition \ref{d179}). Remark that if $B$ 
is a 
point and $\dim Y$ is odd, then the equivariant eta form here degenerates to 
the 
canonical equivariant reduced eta invariant by taking a special perturbation 
operator \cite[Remark 2.20]{Liu2016}.

	Let $i:W\rightarrow V$ be an equivariant embedding of smooth 
	$G$-equivariant oriented manifolds with even codimension. Let 
	$\pi_V:V\rightarrow B$ be a $G$-equivariant submersion with closed 
	fibres $X$, whose restriction $\pi_W:W\rightarrow B$ is an equivariant 
	submersion with closed fibres $Y$. We assume that $G$ acts on $B$ 
	trivially and the normal bundle $N_{Y/X}$ to $Y$ in $X$ has an 
	equivariant Spin$^c$ structure.

	\begin{center}
		\begin{tikzpicture}[>=angle 90]
		\matrix(a)[matrix of math nodes,
		row sep=2em, column sep=2.5em,
		text height=1.5ex, text depth=0.25ex]
		{&Y&W\\
			&X&V&B.\\};
		\path[->](a-1-2) edge node[left]{\footnotesize{$i$}} (a-2-2);
		\path[->](a-1-2) edge (a-1-3);
		\path[->](a-2-2) edge (a-2-3);
		\path[->](a-1-3) edge node[left]{\footnotesize{$i$}} (a-2-3);
		\path[->](a-2-3) edge node[above]{\footnotesize{$\pi_V$}} (a-2-4);
		\path[->](a-1-3) edge node[above]{\footnotesize{$\pi_W$}} (a-2-4);
		\end{tikzpicture}
	\end{center}
	
	Let $\mF_{Y}=(W, \mE_Y, T_{\pi_W}^HW, g^{TY}, h^{\mE_Y}, 
	\nabla^{\mE_Y})$ 
	and $\mF_{X}=(V, \mE_{X}, T_{\pi_V}^HV, g^{TX}, h^{\mE_{X}}, 
	\nabla^{\mE_{X}})$ 
	be two equivariant geometric families over $B$ such that $( 
	T_{\pi_W}^HW, g^{TY})$ and $(T_{\pi_V}^HV, g^{TX})$ satisfy the 
	totally geodesic condition (\ref{bl0960}). 

	We state our main result of this paper as follows.
	\begin{thm}\label{i01}
		Assume that the equivariant geometric families $\mF_Y$ and $\mF_X$ 
		satisfy the fundamental assumption (\ref{bl1011}) and (\ref{bl1013}).
		Let $\mA_Y$ and $\mA_{X}$ be the perturbation operators with respect 
		to 
		$\mF_{Y}$ and $\mF_X$. 
		Then 
		for any compact submanifold $K$ of $B$, there exists $T_0>2$ 
		such 
		that
		for any $T\geq T_0$, modulo exact forms, over $K$, we have
		\begin{multline}\label{i02}
		\tilde{\eta}_g(\mF_X, \mathcal{A}_X)=\tilde{\eta}_g(\mF_Y, 
		\mathcal{A}_Y)+\int_{X^g}\widehat{\mathrm{A}}_g(TX, 
		\nabla^{TX})\,\gamma_g^X(\mF_Y,\mF_X)
		\\
		+\ch_g(\mathrm{sf}_G\{D(\mF_X)+\mA_X, D(\mF_{X}) 
		+T\mathcal{V}+\mA_{T,Y}\})
		.
		\end{multline}
	Here $\gamma_g^X(\mF_Y,\mF_X)$ is the equivariant Bismut-Zhang current 
	defined in Definition \ref{bl0554}.	The last term in (\ref{i02}) is the 
	equivariant chern character of the equivariant Dai-Zhang higher spectral 
	flow which explains the $\mathrm{mod}\, \Z$ term in the original 
	Bismut-Zhang 
	embedding formula.
	\end{thm}

	The proof of our main result here is highly related to the analytical 
	localization technique developed in 
	\cite{MR1188532,MR1316553,Bismut1997,MR2097553}.
		Thanks to the functoriality of equivariant eta forms 
	proved in \cite{Liu2016,MR3626562}, we only need to prove the embedding 
	formula when $B$ is a point and $\dim X$ is odd.
		
	Note that in \cite{MR2129895,MR2532714}, the authors give another proof 
	of the Bismut-Zhang embedding formula without using the analytical 
	localization technique. It is interesting to ask whether there is 
	another proof of our main result here from that line.
	
	Our paper is organized as follows. In Section 1,  we summarize the 
	definition and the properties of equivariant eta forms in \cite{Liu2016} 
	using the language of Clifford modules. In Section 2, we state our main 
	result. In this section, we also discuss an application on 
	the equivariant Atiyah-Hirzebruch direct image. In Section 3, we prove 
	our
	main result in two parts. In Section 3.1, we prove Theorem \ref{i01} 
	when the base space is a point using some intermediary 
	results along the lines of \cite{MR1214954}, the proof of which rely on 
	almost 
	identical arguments of \cite{MR1214954,MR1316553}. In Sections 3.2, we 
	explain how to use the functoriality to reduce Theorem \ref{i01} to the 
	case in Section 3.1. 
	
	To simplify the notations, we use the Einstein summation convention in 
	this paper.
	
	In the whole paper, we use the superconnection formalism of Quillen 
	\cite{MR790678}. 
	
	For a fibre bundle $\pi: V\rightarrow B$, we will often use the 
	integration of the differential forms along the oriented fibres $X$ in 
	this paper.
	Since the fibres may be odd dimensional,
	we must make precise our sign conventions.
	Let $\alpha$ be a differential form on $V$ which in local coordinates is 
	given by
	\begin{align}\label{e01135}
	\alpha=f\cdot \pi^*dy^{1}\wedge \cdots\wedge \pi^*dy^{q}\wedge 
	dx^1\wedge\cdots\wedge dx^n,
	\end{align}
	where $\{dy^p\}$ and $\{dx^i\}$ are the local frames of $T^*B$ and 
	$T^*X$, respectively.
	We set
	\begin{align}\label{e01136}
	\int_X\alpha=dy^{1}\wedge \cdots\wedge dy^{q}\cdot 
	\int_Xf\,  dx^1\wedge\cdots\wedge dx^n.
	\end{align}

\section{Equivariant eta forms}\label{s01}

In this section, we summarize the definition and the properties of 
equivariant eta forms in \cite{Liu2016,MR3626562} using the language of 
Clifford modules. 
Note that locally all manifolds are spin. The proofs of them are the same as 
in the spin case. In Section 1.1, we recall elementary results on Clifford 
algebras. In Section 1.2, we describe the geometry of the fibration and 
recall the Bismut superconnection. In Section 1.3, we define the equivariant 
eta form and state the anomaly formula. In Section 1.4, we explain the 
functoriality of the equivariant eta forms.

\subsection{Clifford algebras}\label{s0101}

Let $E^n$ be an oriented Euclidean vector space, such that $\dim E^n=n$, 
with 
orthonormal basis $\{e_i\}_{1\leq i\leq n}$. Let $C(E^n)$ be the complex 
Clifford algebra of $E^n$
defined by the relations
\begin{align}\label{e01001}
e_ie_j+e_je_i=-2\delta_{ij}.
\end{align}
Sometimes, we also denote by $c(e)$ the element of $C(E^n)$ corresponding to 
$e\in E^n$.

If $e\in E^n$, let $e^*\in (E^n)^*$ correspond to $e$ by the scalar product 
of $E^n$. The exterior algebra $\Lambda(E^n)^*\otimes_{\R}\C$ is a module of 
$C(E^n)$ defined by 
\begin{align}\label{bl0001} 
c(e)\alpha=e^*\wedge\alpha -\iota_{e}\alpha
\end{align} 
for any $\alpha\in\Lambda(E^n)^*\otimes_{\R}\C$. The map $a\mapsto c(a)\cdot 
1$, 
$a\in C(E^n)$,
induces an isomorphism of vector spaces
\begin{align}\label{bl0003}
\sigma:C(E^n)\rightarrow \Lambda(E^n)^*\otimes_{\R} \C.	
\end{align}

If $n$ is even, up to isomorphism, $C(E^n)$ has a unique irreducible module, 
$\mS(E^n)$, which is $\Z_2$-graded. We denote this $\Z_2$-grading of the 
spinor by $\tau$. 
Moreover, there are isomorphisms of $\Z_2$-graded algebras
\begin{align}\label{bl0006}
C(E^n)\simeq \End(\mS(E^n))\simeq \mS(E^n)\widehat{\otimes} \mS(E^n)^*.
\end{align}
 We consider the group $\mathrm{Spin}_n^c$ as a multiplicative subgroup of 
 the group of units of $C(E^n)$.
For the definition and the properties of the group $\mathrm{Spin}_n^c$, see 
\cite[Appendix D]{MR1031992}.
Note that $\mS(E^n)$ is also an irreducible representation of 
$\mathrm{Spin}_n^c$ induced by the Clifford action.

Let $F^m$ be another oriented Euclidean vector space such that $\dim F^m=m$, 
with 
orthonormal basis 
$\{f_p\}_{1\leq p\leq m}$. Then as Clifford algebras,
\begin{align}\label{e01087}
c(F^m\oplus E^n)\simeq c(F^m)\widehat{\otimes}c(E^n).
\end{align}
Moreover, We have $\mS(F^m\oplus 
E^n)\simeq\mS(F^m)\widehat{\otimes}\mS(E^n)$.

If one of $m$ and $n$ is even, we simply assume $m$ is even, the spinor 
$\mS(F^m\oplus E^n)$ is isomorphic to $\mS(F^m)\otimes \mS(E^n)$ with the 
Clifford actions 
\begin{align}\label{bl0988} 
c(f_p)\widehat{\otimes} 1\mapsto c(f_p)\otimes 
1,\quad1\widehat{\otimes}c(e_i)\mapsto \tau\otimes c(e_i).
\end{align}

If $m$ and $n$ are both odd, let
\begin{align}\label{bl0989} 
\Gamma_1=
\left(\begin{array}{cc}
0 & 1 \\
1 & 0
\end{array}
\right),\quad
\Gamma_2=
\left(\begin{array}{cc}
0 & \sqrt{-1} \\
- \sqrt{-1}  & 0
\end{array}
\right).
\end{align}
The spinor $\mS(F^m\oplus E^n)$ is isomorphic to $\mS(F^m)\otimes 
\mS(E^n)\otimes 
\mathbb{C}^2$
with Clifford actions 
\begin{align}\label{bl0990} 
c(f_p)\widehat{\otimes} 1\mapsto c(f_p)\otimes 
1\otimes \Gamma_1,\quad 1\widehat{\otimes}c(e_i)\mapsto 1\otimes 
c(e_i)\otimes \Gamma_2.
\end{align}
The $\Z_2$-grading of $\mS(F^m\oplus E^n)$ under this 
isomorphism is 
\begin{align}\label{bl0991} 
\Id_{\mS(F^m)}\otimes \Id_{\mS(E^n)}\otimes
\left(\begin{array}{cc}
1 & 0 \\
0 & -1
\end{array}
\right).
\end{align}	

Let $\tau^{\Lambda}$ be the canonical $\Z_2$-grading of $\Lambda(F^m)$.
Then $\Lambda(F^m)\widehat{\otimes}\mS(E^n)$ could be regarded as 
$\Lambda(F^m)\otimes \mS(E^n)$ with 
\begin{align}\label{bl0992} 
\alpha\widehat{\otimes} 
1\mapsto \alpha\otimes 1\quad 
1\widehat{\otimes}c(e_i)\mapsto\tau^{\Lambda}\otimes c(e_i)
\end{align}
for
$\alpha\in \Lambda(F^m)$.
Indeed,
\begin{align}\label{bl0993} 
(\alpha\otimes 1)(\tau^{\Lambda}\otimes c(e_i))=-(\tau^{\Lambda}\otimes 
c(e_i))(\alpha\otimes 1).
\end{align}
We abbreiate $c(f_p)\widehat{\otimes} 1$, $\alpha\widehat{\otimes} 1$, 
$1\widehat{\otimes}c(e_i)$ by $c(f_p)$, $\alpha$, $c(e_i)$. Then our 
notation here is the same as the usual one in \cite{Bismut1985}. The purpose 
of writing in this way is 
to explain the extension of some fibrewise operators in the following 
sections to the $\Z_2$-graded 
setting clearly.

\subsection{Bismut superconnection}\label{s0102}

Let $\pi:W\rightarrow B$ be a smooth submersion of smooth oriented 
manifolds with closed fibres $Y$. We assume that $B$ is connected. Remark 
that $W$ here is not assumed to be connected. 

Let $TY=TW/B$ be the relative tangent bundle to the fibres $Y$.
Then $TY$ is orientable.
Let $T_{\pi}^HW$ be a horizontal subbundle of $TW$ such that
\begin{align}\label{e01010}
	TW=T_{\pi}^HW\oplus TY.
\end{align}
The splitting (\ref{e01010}) gives an identification
\begin{align}\label{e01011}
	T_{\pi}^HW\cong \pi^*TB.
\end{align}
If there is no ambiguity, we will omit the subscript $\pi$ in $T_{\pi}^HW$.

Let $g^{TY}$, $g^{TB}$ be Riemannian metrics on $TY$, $TB$. We equip 
$TW=T_{}^HW\oplus TY$ with the Riemannian metric
\begin{align}\label{e01013}
	g^{TW}=\pi^*g^{TB}\oplus g^{TY}.
\end{align}

Let $\nabla^{TW}$, $\nabla^{TB}$ be the Levi-Civita connections on $(W, 
g^{TW})$, $(B, g^{TB})$. Let $P^{TY}$ be the projection
$
P^{TY}:TW=T_{}^HW\oplus TY \rightarrow TY.
$ Set
\begin{align}\label{e01014}
	\nabla^{TY}=P^{TY}\nabla^{TW}P^{TY}.
\end{align}
Then $\nabla^{TY}$ is a Euclidean connection on $TY$.

Let $\nabla^{TB,TY}$ be the connection on $TW=T_{}^HW\oplus TY$ defined by
\begin{align}\label{e01015}
	\nabla^{TB,TY}=\pi^*\nabla^{TB}\oplus\nabla^{TY}.
\end{align}
Then $\nabla^{TB,TY}$ preserves the metric $g^{TW}$ in (\ref{e01013}).

Set
\begin{align}\label{e01017}
S=\nabla^{TW}-\nabla^{TB,TY}.
\end{align}
Then $S$ is a 1-form on $W$ with values in antisymmetric elements of 
$\End(TW)$. By \cite[Theorem 1.9]{Bismut1985}, we know that $\nabla^{TY}$ 
and the $(3,0)$-tensor $g^{TW}(S(\cdot)\cdot,\cdot)$ only depend on 
$(T^HW,g^{TY})$.

Let $C(TY)$ be the Clifford algebra bundle of $(TY, g^{TY})$, whose fibre at 
$x\in W$ is the Clifford algebra
$C(T_xY)$ of the Euclidean vector space $(T_xY, g^{T_xY})$. 
A $\Z_2$-graded self-adjoint $C(TY)$-module,
\begin{align}\label{bl0959} 
\mE=\mE_+\oplus\mE_-,
\end{align}
 is a $\Z_2$-graded vector bundle equipped with a Hermitian metric $h^{\mE}$ 
 preserving the splitting 
 (\ref{bl0959}) and a 
 fiberwise Clifford multiplication $c$ of $C(TY)$ such that the action $c$ 
 restricted to $TY$ is skew-adjoint on $(\mE,h^{\mE})$.
Let $\tau^{\mE}$ be the $\Z_2$-grading of $\mE$ which is $\pm 1$ on 
$\mE_{\pm}$.

For a locally oriented orthonormal 
basis $e_1,\cdots,e_n$ of $TY$, we define the chirality operator on $\mE$ by
\begin{align}\label{bl0926} 
\Gamma=
\left\{
\begin{array}{ll}
(\sqrt{-1})^{n/2}c(e_1)\cdots c(e_n), & \hbox{if $n$ is even;} \\
\quad\quad\quad\quad \Id_{\mE}, & \hbox{if $n$ is odd.}
\end{array}
\right.
\end{align}
Note that our definition here is different from \cite[Lemma 3.17]{MR2273508} 
when $n$ is odd.
Then $\Gamma$ does not depend on the choice of the basis and is globally 
defined. 
We note that $\Gamma^2=1$ and $[\tau^{\mE},\Gamma]=0$. 
Set
\begin{align}\label{bl0927}
	\tau^{\mE/\mS}=\tau^{\mE}\cdot\Gamma.
\end{align}
Then $\left(\tau^{\mE/\mS}\right)^2=1$.

Locally, we could write
\begin{align}\label{bl0994} 
\mE=\mS_0(TY)\widehat{\otimes}\xi_{\pm},
\end{align}
where $\mS_0(TY)$ is the spinor bundle for the (possibly non-existent) spin 
structure of $TY$ and $\xi_{\pm}$ is a $\Z_2$-graded vector bundle. Then 
$\Gamma$, $\tau^{\mE/\mS}$ and $\tau^{\mE}$ correspond to the 
$\Z_2$-gradings of $\mS_0(TY)$, $\xi_{\pm}$ and 
$\mS_0(TY)\widehat{\otimes}\xi_{\pm}$.

Let $\nabla^{\mE}$ be a Clifford connection on $\mE$ 
associated with $\nabla^{TY}$, that is,
$\nabla^{\mE}$ preserves $h^{\mE}$ and the splitting (\ref{bl0959}) and 
for any $U\in TW$, $Z\in \cC^{\infty}(W,TY)$,
\begin{align}\label{e01024}
	\left[\nabla_U^{\mE}, c(Z)\right]=c\left(\nabla^{TY}_UZ\right).
\end{align}

Let $G$ be a compact Lie group which acts on $W$ and $B$ such that for any 
$g\in G$, $\pi\circ g=g\circ\pi$.
We assume that the action of $G$ preserves the splitting (\ref{e01010}) and 
the orientation of $TY$ and could be lifted on $\mE$ such that it is 
compatible with the Clifford action and preserves the splitting 
(\ref{bl0959}).
We assume that $g^{TY}$, $h^{\mE}$, $\nabla^{\mE}$ are $G$-invariant.

\begin{defn}\label{d001}(Compare with 
\cite[Definition 2.2]{Bunke2009}, \cite[Definition 1.1]{Liu2016})
	An equivariant geometric family $\mF$ over $B$ is a family of 
	$G$-equivariant geometric data
	\begin{align}\label{d002}
		\mF=(W, \mE, T_{}^HW, g^{TY}, h^{\mE}, \nabla^{\mE})
	\end{align}
	described as above. We call the equivariant geometric family $\mF$ is 
	even (resp. odd) if
	for any connected component of fibres, the dimension of it is even 
	(resp. odd). 
\end{defn}

Let $D(\mF)$ be the fiberwise Dirac operator
\begin{align}\label{e01029}
D(\mF)=c(e_i)\nabla_{e_i}^{\mE}
\end{align}
associated with the equivariant geometric family $\mF$. Then the $G$-action 
commutes with $D(\mF)$.

For $b\in B$, let $\cE_{b}$ be the set of smooth sections over $Y_b$ of 
$\mE_b$. As in \cite{Bismut1985},
we will regard $\cE$ as
an infinite dimensional fibre bundle over $B$.
If $V\in TB$, let $V^H\in T_{}^HW$ be its horizontal lift in $T_{}^HW$ so 
that $\pi_*V^H=V$.
For any $V\in TB$, $s\in \cC^{\infty}(B,\cE)=\cC^{\infty}(W,\mE)$,
by \cite[Proposition 1.4]{MR853982}, the connection
\begin{align}\label{e01028}
	\nabla_V^{\cE,u}s:=\nabla_{V^H}^{\mE}s-\frac{1}{2}\la S(e_i)e_i, V^H \ra 
	\,s
\end{align}
preserves the $L^2$-product on $\cE$.
Let $\{f_p\}$ be a local orthonormal frame of $TB$ and $\{f^p\}$ be its dual.
We denote by $\nabla^{\cE,u}=f^p\wedge \nabla^{\cE,u}_{f_p}$.
Let $T$ be the torsion of $\nabla^{TB,TY}$.
We denote by
$
	c(T)=\frac{1}{2}\,c\left(T(f_p^H, f_q^H)\right)f^p\wedge f^q\wedge.
$
By \cite[(3.18)]{Bismut1985}, the rescaled Bismut superconnection
$
	\mathbb{B}_u:\cC^{\infty}(B,\Lambda(T^*B)\widehat{\otimes}\cE)\rightarrow\cC^{\infty}(B,
	 \Lambda(T^*B)\widehat{\otimes}\cE)
$
is defined by
\begin{align}\label{e01042}
	\mathbb{B}_u=\sqrt{u}D(\mF)+\nabla^{\cE,u}-\frac{1}{4\sqrt{u}}c(T).
\end{align}
Obviously, the Bismut superconnection $\mathbb{B}_u$ commutes with the 
$G$-action.
Furthermore, $\mathbb{B}_u^2$ is a $2$-order elliptic differential operator 
along the fibres $Y$.
Let $\exp(-\mathbb{B}_u^2)$ be the family of heat operators associated with 
the fiberwise elliptic operator $\mathbb{B}_u^2$.
From \cite[Theorem 9.51]{MR2273508}, we know that $\exp(-\mathbb{B}_u^2)$ is 
a smooth family of smoothing operators.

Let $P$ be a section of $\Lambda(T^*B)\widehat{\otimes}\End(\cE)$.
Set
\begin{align}\label{bl1008} 
	\tr_s[P]:=\tr[\tau^{\mE} P].
\end{align}
Here the trace operator on the right hand side of (\ref{bl1008}) only acts 
on $\mE$ and takes values in $\Lambda(T^*B)$.
We denote by $\tr^{\mathrm{odd/even}}_s[P]$ the part of $\tr_s[P]$
which takes values in odd or even forms.
We use the convention that if $\omega\in \Lambda(T^*B)$, 
\begin{align}\label{e01056}
	\tr_s[\omega P]=\omega\tr_s[P].
\end{align}
It is compatible with the sign convention (\ref{e01136}).
Set
\begin{align}\label{i16}
	\widetilde{\tr}[P]=
	\left\{
	\begin{array}{ll}
		\tr_s[P], & \hbox{if $\dim Y$ is even;} \\
		\tr_s^{\mathrm{odd}}[P], & \hbox{if $\dim Y$ is odd.}
	\end{array}
	\right.
\end{align}

\subsection{Equivariant eta forms}\label{s0103}

In this subsection, we state the definition and the anomaly formula of 
equivariant eta forms in the language of Clifford modules.

We assume that $G$ acts trivially on $B$.

Take $g\in G$ and set
$
W^g=\{x\in W: gx=x\}.
$
Then $W^g$ is a submanifold of $W$ and $\pi|_{W^g}:W^g\rightarrow B$ is a 
fibre bundle with closed fibres $Y^g$.
Let $N_{W^g/W}$ denote the normal bundle of $W^g$ in $W$, then 
$N_{W^g/W}=TY/TY^g$.
We also denote it by $N_{Y^g/Y}$.

We denote the differential of $g$ by $dg$
which gives a bundle isometry $dg: N_{Y^g/Y}\rightarrow N_{Y^g/Y}$. Since 
$g$ lies in a compact abelian Lie group,
we know that there is an orthonormal decomposition of smooth vector bundles 
over $W^g$
\begin{align}\label{e01046}
TY|_{W^g}=TY^g\oplus N_{Y^g/Y}=TY^g\oplus \bigoplus_{0<\theta\leq 
\pi}N(\theta),
\end{align}
where $dg|_{N(\pi)}=-\mathrm{id}$ and for each $\theta$, $0<\theta<\pi$, 
$N(\theta)$ is a complex vector bundle
on which $dg$ acts by multiplication by $e^{i\theta}$. Since $g$ preserves 
the orientation of $TY$ and $\det(dg|_{N(\pi)})=1$, by the property of 
isometry, $\dim N(\pi)$ is even. So the normal bundle $N_{Y^g/Y}$ is even 
dimensional.

Observe that if $N(\pi)=0$ or if $TY$ has a $G$-equivariant Spin$^c$ 
structure, then $TY^g$ is canonically oriented (cf. \cite[Proposition 
6.14]{MR2273508}, \cite[Proposition 2.1]{MR3626562}). In general, $TY^g$ is 
not necessary oriented.
For simplicity, in this paper we assume that $TY^g$ is oriented. In this 
case $N(\pi)$ is oriented. We fix an orientation of $TY^g$ which is induced 
by the orientations of $N_{\theta}$ and $TY$.

We remark here that if the fixed point sets are not oriented, we could also 
get the formulas in this paper in the sense of Berezin integral as in 
\cite[Theorem 6.16]{MR2273508}.

Let $E$ be an equivariant real Euclidean vector bundle over $W$.
We could get
the decomposition of real vector bundles over $W^g$ in the same way as 
(\ref{e01046}),
\begin{align}\label{e01047}
E|_{W^g}=\bigoplus_{0\leq\theta\leq \pi}E(\theta).
\end{align}
Here we also denote $E(0)$ by $E^g$. 

Let $\nabla^E$ be an equivariant Euclidean connection on $E$. Then it 
preserves the decomposition (\ref{e01047}).
Let $\nabla^{E^g}$ and $\nabla^{E(\theta)}$ be the corresponding induced 
connections on $E^g$ and $E(\theta)$,
and let $R^{E^g}$ and $R^{E(\theta)}$ be the corresponding curvatures.

Set
\begin{multline}\label{e01051}
\widehat{\mathrm{A}}_g(E,\nabla^{E})=\mathrm{det}^{\frac{1}{2}}\left(\frac{\frac{\sqrt{-1}}{4\pi}R^{E^g}}{\sinh
 \left(\frac{\sqrt{-1}}{4\pi}R^{E^g}\right)}\right)
\\
\cdot
\prod_{0<\theta\leq \pi}\left(\sqrt{-1}^{\frac{1}{2}\dim 
E(\theta)}\mathrm{det}^{\frac{1}{2}}\left(1-g
\exp\left(\frac{\sqrt{-1}}{2\pi}R^{E(\theta)}\right)\right)\right)^{-1}.
\end{multline}


Let $\End_{C(TY)}(\mE)$ be the set of endomorphisms of $\mE$ commuting with 
the Clifford action. Then it is a vector bundle over $W$.
As in \cite[Definition 3.28]{MR2273508}, for any $a\in \End_{C(TY)}(\mE)$, 
we define the relative trace $\tr^{\mE/\mS}:\End_{C(TY)}(\mE)\rightarrow \C$ 
by
\begin{align}\label{bl1009}
\tr^{\mE/\mS}[a]=
\left\{
\begin{array}{ll}
2^{-n/2}\tr_s[\Gamma a], & \hbox{if $n=\dim Y$ is even;} \\
2^{-(n-1)/2}\tr_s[a], & \hbox{if $n=\dim Y$ is odd.}
\end{array}
\right.
\end{align}

Let $R^{\mE}$ be the curvature of $\nabla^{\mE}$.
Let
\begin{multline}\label{e01038}
R^{\mE/\mS}:=R^{\mE}-\frac{1}{4}\la R^{TY}e_i, e_j\ra c(e_i)c(e_j)
\\
\in 
\cC^{\infty}(W, \pi^*\Lambda(T^*B)\otimes \End_{C(TY)}(\mE))
\end{multline}
be the twisting curvature of the $C(TY)$-module $\mE$ as in 
\cite[Proposition 3.43]{MR2273508}. 

By \cite[Lemma 6.10]{MR2273508}, along $W^g$, the action of $g\in G$ on 
$\mE$ may be identified with a section $g^{\mE}$ of 
$C(N_{Y^g/Y})\otimes_{\C} 
\End_{C(TY)}(\mE)$.
Let $\dim N_{Y^g/Y}=\ell_1$.
Under the isomorphism (\ref{bl0003}), $\sigma(g^{\mE})\in \cC^{\infty}(W^g, 
(\Lambda N_{Y^g/Y}^*\otimes_{\R}\C)\otimes_{\C} \End_{C(TY)}(\mE))$. Since 
we assume that 
$N_{Y^g/Y}$ is oriented, paring with the volume form, we could get the 
highest degree coefficient $\sigma_{\ell_1}(g^{\mE})\in \cC^{\infty}(W^g,  
\End_{C(TY)}(\mE))$ of $\sigma(g^{\mE})$.
 
Then we could define the localized relative Chern character 
$\ch_g(\mE/\mS,\nabla^{\mE})\in \Omega^*(W^g,\C)$ in the same way as 
\cite[Definition 6.13]{MR2273508} by
\begin{align}\label{bl1010} 
	\ch_g(\mE/\mS, 
	\nabla^{\mE}):=\frac{2^{\ell_1/2}}{\mathrm{det}^{1/2}(1-g|_{N_{Y^g/Y}})}\tr^{\mE/\mS}\left[\sigma_{\ell_1}(g^{\mE})\exp\left(-\frac{R^{\mE/\mS}|_{W^g}}{2\pi\sqrt{-1}}\right)\right].
\end{align}
Note that if $TY$ has an equivariant spin structure, the localized relative 
Chern character here is just the usual equivariant Chern character. 

Recall that if $B$ is compact, the equivariant $K$-group $K_G^0(B)$ is the 
Grothendieck group of the equivalent classes of the equivariant vector 
bundles over $B$. Let $\iota:B\rightarrow B\times S^1$ be a $G$-equivariant 
inclusion map. It is well known that if the $G$-action on $S^1$ is trivial,
\begin{align}\label{d124}
K_G^1(B)\simeq \ker\left(\iota^*:K_G^0(B\times S^1)\rightarrow 
K_G^0(B)\right).
\end{align}
For $x\in K_G^0(B)$, $g\in G$, the classical equivariant Chern character map 
sends $x$ to $\ch_g(x)\in H^{\mathrm{even}}(B,\C)$.
By (\ref{d124}), for $x\in K_G^1(B)$, we can regard $x$ as an element $x'$ 
in $K_G^0(B\times S^1)$. The odd equivariant Chern character map
\begin{align}\label{d130}
\ch_g:K_G^1(B)\longrightarrow H^{\mathrm{odd}}(B,\C)
\end{align}
is defined by (cf. \cite[(2.52)]{MR3626562})
\begin{align}\label{d131}
\ch_g(x):=\int_{S^1}\ch_g(x').
\end{align}
We adopt the sign notation in the integral as in (\ref{e01136}).

Furthermore, the classical construction of Atiyah-Singer
assigns to each equivariant geometric family $\mF$ its equivariant (analytic)
index $\ind(D(\mF))\in K_G^*(B)$ \cite{MR0285033,MR0279833} ($*=0$ or $1$ 
corresponds to $\mF$ is even or odd).

For $\alpha\in \Omega^j(B)$, set
\begin{align}\label{e01059}
\psi_B(\alpha)=\left\{
\begin{array}{ll}
\left(\frac{1}{2\pi \sqrt{-1}}\right)^{\frac{j}{2}}\cdot \alpha, & \hbox{if 
$j$ is even;} \\
\frac{1}{\sqrt{\pi}}\left(\frac{1}{2\pi 
\sqrt{-1}}\right)^{\frac{j-1}{2}}\cdot \alpha, & \hbox{if $j$ is odd.}
\end{array}
\right.
\end{align}

\begin{thm}\label{e01061}\cite[Theorem 2.2]{MR3626562}
	For any $u>0$ and $g\in G$, the differential form 
	$\psi_B\widetilde{\tr}[g\exp(-\mathbb{B}_u^2)]\in \Omega^{*}(B, \C)$ is 
	closed and its cohomology class
	is independent of $u$. As $u\rightarrow 0$, 
	\begin{align}\label{e01062}
	\lim_{u\rightarrow 
	0}\psi_B\widetilde{\tr}[g\exp(-\mathbb{B}_u^2)]=\int_{Y^g}\widehat{\mathrm{A}}_g(TY,\nabla^{TY})\,
	\ch_g(\mE/\mS,\nabla^{\mE}).
	\end{align}
	If $B$ is compact,  the differential form 
	$\psi_B\widetilde{\tr}[g\exp(-\mathbb{B}_u^2)]$
	represents $\ch_g(\ind(D(\mF)))$.

\end{thm}

\begin{defn}\label{d179}\cite[Definition 2.10]{Liu2016}
	A perturbation operator with respect to $D(\mF)$, denoted by $\mA$, is 
	defined to be a smooth family of $G$-equivariant bounded self-adjoint 
	pseudodifferential 
	operators on $\mE$ along the fibres such that 
	it commutes (resp. anti-commutes) with the $\Z_2$-grading of 
	$\mE$ when the fibres are odd (resp. even) dimensional, and
	$D(\mF)+\mA$ is invertible.
\end{defn}

Remark that from \cite[Proposition 2.3]{Liu2016}, 
	if $B$ is compact and at least one component of the fibres has the 
	non-zero dimension, then there exists a perturbation operator with 
	respect to $D(\mF)$ if and only if $\ind(D(\mF))=0\in K_G^*(B)$. 

In the followings, we always assume that
	there exists a perturbation operator with respect to $D(\mF)$ on $\mF$.

For $\alpha\in \Lambda(T^*(\R\times B))$, we can expand $\alpha$ in the 
form
\begin{align}\label{e01003}
\alpha(u)=du\wedge \alpha_0(u)+\alpha_1(u),\quad \alpha_0(u), \alpha_1(u)\in 
\Lambda (T^*B).
\end{align}
Set
\begin{align}\label{e01005}
[\alpha(u)]^{du}:=\alpha_0(u).
\end{align}

Let $\chi\in \cC_0^{\infty}(\R)$ be a cut-off function such that
\begin{align}\label{d013}
\chi(u)=
\left\{
\begin{aligned}
&0,  &\hbox{if $u\leq 1$;} \\
&1,  &\hbox{if $u\geq 2$.}
\end{aligned}
\right.
\end{align}

Let $\mA$ be a perturbation operator with respect to $D(\mF)$. Then $\mA$ 
could be extended to $1\widehat{\otimes}\mA$ on 
$\cC^{\infty}(B,\pi^*\Lambda(T^*B)\widehat{\otimes}\mE)$ as in 
(\ref{bl0992}).	
Explicitly, the extended perturbation operator $1\widehat{\otimes}\mA$ which 
acts along 
the fibres $Y$ on 
$\cC^{\infty}(B,\pi^*\Lambda(T^*B)\widehat{\otimes}\mE)$ is considered as 
$\tau^{\Lambda}\otimes\mA$
on $\cC^{\infty}(B,\pi^*\Lambda(T^*B)\otimes\mE)$ and 
$\alpha\widehat{\otimes} 1\mapsto \alpha\otimes 1$.
Then as in (\ref{bl0993}), we have
\begin{align}\label{bl0958} 
(\alpha\widehat{\otimes} 1)(1\widehat{\otimes} \mA)=-(1\widehat{\otimes} 
\mA)(\alpha\widehat{\otimes} 1).
\end{align}
We usually abbreviate $1\widehat{\otimes}\mA$ by $\mA$ when there is no 
confusion. 
Set
\begin{align}\label{d070}
\mathbb{B}_u'=\mathbb{B}_u+\sqrt{u}\chi(\sqrt{u})\mA.
\end{align}

\begin{defn}\label{d017}\cite[Definition 2.11]{Liu2016} 
	For any $g\in G$, modulo exact forms, the equivariant Bismut-Cheeger eta 
	form with perturbation operator $\mA$ is defined by
	\begin{multline}\label{i17}
	\tilde{\eta}_g(\mF, \mA)
	:=
	-\int_0^{\infty}\left\{\psi_{\R\times 
	B}\left.\widetilde{\tr}\right.\left[g\exp\left(-\left(
	\mathbb{B}_{u}'+du\wedge\frac{\partial}{\partial 
	u}\right)^2\right)\right]\right\}^{du}du
	\\	\in \Omega^*(B,\C)/\Im\, d.
	\end{multline}
As in (\ref{bl0993}) and (\ref{bl0958}), we adopt the convention that $du$ 
anti-commutes with $\mA$ and $c(v)$ for any $v\in TY$. 	
\end{defn}

From the discussion in \cite[Section 2.3]{Liu2016},
the equivariant eta form with perturbation in Definition \ref{d017} is well 
defined and does not depend on the cut-off function.
Moreover, since we assume that $Y^g$ is oriented,
we have (cf. \cite[(2.44)]{Liu2016})
\begin{align}\label{e01085}
d^B\tilde{\eta}_g(\mF,\mA)=
\int_{Y^g}\widehat{\mathrm{A}}_g(TY,\nabla^{TY})\,
\ch_g(\mE/\mS,\nabla^{\mE}).
\end{align}

\begin{rem}\label{e01067}
	After changing the variable, we have
	\begin{align}\label{e01070}
	\tilde{\eta}_g(\mF,\mA)=-\int_0^{\infty}\left\{\psi_{\R\times 
	B}\left.\widetilde{\tr}\right.\left[g\exp\left(-\left(
	\mathbb{B}_{u^2}'+du\wedge\frac{\partial}{\partial 
	u}\right)^2\right)\right]\right\}^{du}du.
	\end{align}
	We will often use this formula as the definition of the equivariant eta 
	form.
\end{rem}

	Explicitly, 
	\begin{align}\label{i19}
	\tilde{\eta}_g(\mF, \mA)=
	\left\{
	\begin{aligned}
	& 
	\int_0^{\infty}\left.\frac{1}{\sqrt{\pi}}\psi_{B}\tr_s^{\mathrm{even}}\right.\left[g\left.\frac{\partial
	 \mathbb{B}_{u^2}'}{\partial u}\right.
	\exp(-(\mathbb{B}_{u^2}')^{2})\right] du 
	\\
	& \quad\quad\quad\quad\quad\quad\quad\quad\quad\in 
	\Omega^{\mathrm{even}}(B,\C)/\Im\, d,
	\hbox{if $\mF$ is odd;} \\
	&\int_0^{\infty} \left.\frac{1}{2 
	\sqrt{\pi}\sqrt{-1}}\psi_{B}\tr_s\right.\left[g\left.\frac{\partial 
	\mathbb{B}_{u^2}'}{\partial u}\right.
	\exp(-(\mathbb{B}_{u^2}')^{2})\right] du	
	\\
	& \quad\quad\quad\quad\quad\quad\quad\quad\quad\in 
	\Omega^{\mathrm{odd}}(B,\C)/\Im\, d,
	\hbox{if $\mF$ is even.} \\
	\end{aligned}
	\right.
	\end{align}
	From \cite[Remark 2.20]{Liu2016},
	when $B$ is a point, $\dim Y$ is odd, letting $\mA=P_{\ker D(\mF_Y)}$ be 
	the orthogonal projection onto the kernel of $D(\mF_Y)$, the equivariant 
	eta form $\tilde{\eta}_g(\mF, \mA)$ is just the equivariant reduced eta 
	invariant defined in \cite{MR511246}.
	Note that from (\ref{i19}), if $B$ is a point and $\dim Y$ is even, we 
	have
	$\tilde{\eta}_g(\mF, \mA)=0$ for any perturbation operator $\mA$.

	Let $\mF=(W, \mE, T_{}^HW, g^{TY}, h^{\mE}, \nabla^{\mE})$ and $\mF'=(W, 
	\mE, T_{}^{'H}W, g^{'TY}, h^{'\mE}, \nabla^{'\mE})$ be two equivariant 
	geometric families over $B$.
Let $$\left(\widetilde{\widehat{\mathrm{A}}}_g\cdot
\widetilde{\ch_g}\right)(\nabla^{TY}, \nabla^{'TY}, \nabla^{\mE}, 
\nabla^{'\mE})\in \Omega^*(W^g,\C)/\Im d$$ be the Chern-Simons 
form (cf. \cite[Appendix B]{MR2339952}) such that
\begin{multline}\label{e01110}
d \left(\widetilde{\widehat{\mathrm{A}}}_g\cdot
\widetilde{\ch_g}\right)(\nabla^{TY}, \nabla^{'TY}, \nabla^{\mE}, 
\nabla^{'\mE})
\\
= \widehat{\mathrm{A}}_g(TY,\nabla^{'TY})\,
\ch_g(\mE/\mS, \nabla^{'\mE})-\widehat{\mathrm{A}}_g(TY,\nabla^{TY})\,
\ch_g(\mE/\mS, \nabla^{\mE}).
\end{multline}

When $B$ is compact,
let
$\mathrm{sf}_G\{(D(\mF')+\mA', P'), (D(\mF)+\mA, P)\}\in K_G^*(B)
$, which we often simply denote by $\mathrm{sf}_G\{D(\mF')+\mA', 
D(\mF)+\mA\}$, be the equivariant Dai-Zhang higher spectral flow defined in 
\cite[Definition 2.5, 2.6]{Liu2016},  where $P$, $P'$ are the orthonormal 
projections onto the eigenspaces of positive eigenvalues with respect to 
$D(\mF)+\mA$, $D(\mF')+\mA'$ respectively.
If $B$ is a point and $\dim Y$ is odd, it is just the canonical equivariant 
spectral flow.

 The following anomaly formula is proved in \cite[Theorem 2.17]{Liu2016} and 
 \cite[Theorem 2.7]{MR3626562}.

\begin{thm}\label{e01105}
	Let $\mA, \mA'$ be perturbation operators with respect to $D(\mF)$, 
	$D(\mF')$ respectively.
	For any $g\in G$, modulo exact forms, we have
	
	(a) if $B$ is compact, then
	\begin{multline}\label{e01106}
	\tilde{\eta}_g(\mF', \mA')-\tilde{\eta}_g(\mF,\mA)
	=\int_{Y^g}\left(\widetilde{\widehat{\mathrm{A}}}_g\cdot
	\widetilde{\ch_g}\right)\left(\nabla^{TY}, \nabla^{'TY}, \nabla^{\mE}, 
	\nabla^{'\mE}\right)
	\\
	+\ch_g(\mathrm{sf}_G\{D(\mF')+\mA', D(\mF)+\mA\});
	\end{multline}
	
	(b) if $B$ is noncompact and there exists a smooth path  
	$(\mF_s,\mA_s)$, $s\in [0,1]$, connecting $(\mF,\mA)$ and $(\mF',\mA')$ 
	such that for any $s\in [0,1]$, $\mA_s$ is the perturbation operator of 
	$D(\mF_s)$, then
	\begin{align}\label{bl0951}
	\tilde{\eta}_g(\mF', \mA')-\tilde{\eta}_g(\mF,\mA)
	=\int_{Y^g}\left(\widetilde{\widehat{\mathrm{A}}}_g\cdot
	\widetilde{\ch_g}\right)\left(\nabla^{TY}, \nabla^{'TY}, \nabla^{\mE}, 
	\nabla^{'\mE}\right).
	\end{align}
	
\end{thm}

	\subsection{Functoriality}\label{s0104}
	
	Let $\pi_M:U\rightarrow W$ be a $G$-equivariant submersion of smooth  
	manifolds with closed oriented fibres $M$.
	Let $(\mE_M,h^{\mE_M})$ be a $\Z_2$-graded self-adjoint equivariant 
	$C(TM)$-module.
	Let
	\begin{align}\label{d155}
	\mF_M=(U, \mE_M, T_{\pi_M}^HU, g^{TM}, h^{\mE_M}, \nabla^{\mE_M})
	\end{align}
	be a $G$-equivariant geometric family over $W$. Then $\pi_Z:=\pi\circ 
	\pi_M: U\rightarrow B$ is a $G$-equivariant submersion with closed 
	 fibres $Z$, the orientation of which is the composition of the 
	 orientations of $Y$ and $M$.
	Then we have the diagram of submersions:
	
	\begin{center}\label{e02001}
		\begin{tikzpicture}[>=angle 90]
		\matrix(a)[matrix of math nodes,
		row sep=2em, column sep=2.5em,
		text height=1.5ex, text depth=0.25ex]
		{M&Z&U\\
			&Y&W&B.\\};
		\path[->](a-1-1) edge (a-1-2);
		\path[->](a-1-2) edge node[left]{\footnotesize{}} (a-2-2);
		\path[->](a-1-2) edge (a-1-3);
		\path[->](a-2-2) edge (a-2-3);
		\path[->](a-1-3) edge node[left]{\footnotesize{$\pi_M$}} (a-2-3);
		\path[->](a-2-3) edge node[above]{\footnotesize{$\pi$}} (a-2-4);
		\path[->](a-1-3) edge node[above]{\footnotesize{$\pi_Z$}} (a-2-4);
		\end{tikzpicture}
	\end{center}
	
	Set $T_{\pi_M}^HZ:=T_{\pi_M}^HU\cap TZ$. Then we have the splitting of 
	smooth vector bundles over $U$,
	\begin{align}\label{e02003}
	TZ=T_{\pi_M}^HZ\oplus TM,
	\end{align}
	and
	\begin{align}\label{e02002}
	T_{\pi_M}^HZ\cong \pi_M^*TY.
	\end{align}
	
	Take the geometric data 
	$(T_{\pi_Z}^HU, g_T^{TZ})$ of $\pi_Z$ such that $T_{\pi_Z}^HU\subset 
	T_{\pi_M}^HU$, 
	\begin{align}\label{bl1050}
	g_T^{TZ}=\pi_M^*g^{TY}\oplus \frac{1}{T^2}g^{TM}
	\end{align}
	 and $g^{TZ}=g_1^{TZ}$. 
	We denote the Clifford algebra bundle  with respect to $g_T^{TZ}$ by 
	$C_T(TZ)$
	and the corresponding  1-form in (\ref{e01017}) by $S_{T}$.

	Let $\{e_i \}$, $\{f_p\}$ be local orthonormal frames of $TM$, $TY$ with 
	respect to $g^{TM}$, $g^{TY}$ respectively. Now $\{Te_i\}$ is a local 
	orthonormal frame of $TM$ with respect to the rescaled metric 
	$T^{-2}g^{TM}$. Let $f_p^H$ be the horizontal lift of $f_p$ with respect 
	to (\ref{e02003}).
	Now we define a Clifford algebra homomorphism
	\begin{align}\label{e201}
	G_T:(C_T(TZ),g_T^{TZ})\rightarrow (C(TZ),g^{TZ})
	\end{align}
	by
	$
	G_T(c_T(f_{p}^H))=c(f_{p}^H)$ and $G_T(c_T(Te_i))=c(e_i).
	$
	Under this homomorphism,
	\begin{align}\label{d156}
	\mE_Z:=\pi_M^*\mE_Y\widehat{\otimes} \mE_M
	\end{align}
	with induced Hermitian metric $h^{\mE_Z}$ is a $\Z_2$-graded 
	self-adjoint equivariant $C_T(TZ)$-module. 
	 .

Let 
\begin{align}\label{bl1015} 
\,^0\nabla^{\mE_Z}:=\pi_M^*\nabla^{\mE_Y}\otimes 1+1\otimes \nabla^{\mE_M}.
\end{align}
Then it is a Clifford connection on $\mE_Z$ associated with
\begin{align}
	\nabla^{TY,TM}:=\pi_M^*\nabla^{TY}\otimes 1+1\otimes \nabla^{TM}.
\end{align}	
	Now, we denote the Levi-Civita connection on $TZ$ with respect to 
	$g_{T}^{TZ}$ by $\nabla_T^{TZ}$.
Then we could calculate that
	\begin{multline}\label{e197}
	\nabla^{\mathcal{E}_Z}_T:=\,^0\nabla^{\mE_Z}+\frac{1}{2}\la S_{T} 
	Te_i,f_{p}^H\ra_Tc_T(Te_i)c(f_{p}^H)
	\\
	+\frac{1}{4}\la S_{T}f_{p}^H,f_{q}^H\ra_Tc(f_{p}^H)c(f_{q}^H)
	\end{multline}
	is a Clifford connection associated with $\nabla_T^{TZ}$, where 
	$\la\cdot,\cdot\ra_T=g_T^{TZ}(\cdot,\cdot)$ (cf. 
	\cite[(4.3)]{MR3626562}).
	Thus we get a rescaled equivariant geometric family
	\begin{align}\label{d073}
	\mF_{Z,T}:=(U, \mE_Z, T_{\pi_Z}^HU, g_T^{TZ}, h^{\mE_Z}, 
	\nabla_T^{\mE_Z})
	\end{align}
	over $B$. We write $\mF_Z=\mF_{Z,1}$.  
	
	Let $\mA_M$ be a perturbation 
	operator with respect to  $D(\mF_M)$. 
	Then $\mA_M$ could be extended to $1\widehat{\otimes}\mA_M$ on 
	$\cC^{\infty}(U, 
	\pi_Z^*\Lambda(T^*B)\widehat{\otimes}\pi_M^*\mE_Y
	\widehat{\otimes}\mE_M)$ in the same way as $D(\mF_M)$.
	
	Explicitly, if $\dim Y$ is even, the extended perturbation operator 
	$1\widehat{\otimes}\mA_M$ which 
	acts along the fibres $M$ on $\cC^{\infty}(U, 
	\pi_Z^*\Lambda(T^*B)\widehat{\otimes}\pi_M^*\mE_Y
	\widehat{\otimes}\mE_M)$
	 is considered as $\tau^{\Lambda}\otimes \tau\otimes \mA_M$ on 
	$\cC^{\infty}(U, \pi_Z^*\Lambda(T^*B)\otimes \pi_M^*\mE_Y\otimes\mE_M)$ 
	and $\alpha\widehat{\otimes}1\widehat{\otimes}1\mapsto \alpha\otimes 
	1\otimes 1,\quad 1\widehat{\otimes} c(f_p)\widehat{\otimes}1\mapsto 
	\tau^{\Lambda}\otimes c(f_p)\otimes 1.$
	
	If $\dim Y$ is odd and $\dim M$ is even, 
	$1\widehat{\otimes}\mA_M$ 
	 is considered as $\tau^{\Lambda}\otimes 1\otimes \mA_M$ 
		and 
	$\alpha\widehat{\otimes}1\widehat{\otimes}1\mapsto \alpha\otimes 
	1\otimes 1,\quad 1\widehat{\otimes} c(f_p)\widehat{\otimes}1\mapsto 
	\tau^{\Lambda}\otimes c(f_p)\otimes \tau.$
	
	If $\dim Y$ and $\dim M$ are odd, 
	$1\widehat{\otimes}\mA_M$ 
		 is considered as $\tau^{\Lambda}\otimes 1\otimes \mA_M\otimes 
	 \Gamma_2$ 
	 on 
	 $\cC^{\infty}(U, \pi_Z^*\Lambda(T^*B)\otimes 
	 \pi_M^*\mE_Y\otimes\mE_M\otimes \C^2)$
		and 
	$\alpha\widehat{\otimes}1\widehat{\otimes}1\mapsto \alpha\otimes 
	1\otimes 1\otimes \Id,\quad1\widehat{\otimes} 
	c(f_p)\widehat{\otimes}1\mapsto \tau^{\Lambda}\otimes c(f_p)\otimes 
	1\otimes \Gamma_1.$
	
	We abbreiate $\alpha\widehat{\otimes}1\widehat{\otimes}1$, 
	$1\widehat{\otimes} c(f_p)\widehat{\otimes}1$ by  $\alpha$, $c(f_p)$. 
	Then
	\begin{align}\label{bl0995} 
	\alpha\cdot 1\widehat{\otimes}\mA_M =-1\widehat{\otimes}\mA_M\cdot 
	\alpha,\quad c(f_p)\cdot 
	1\widehat{\otimes}\mA_M=-1\widehat{\otimes}\mA_M\cdot 
	c(f_p).
	\end{align}

	 In \cite[Lemma 2.15]{Liu2016}, we prove that for any compact 
	 submanifold 
	 $K$ of $B$, there exists $T_0>0$ such that
	for $T\geq T_0$, $1\widehat{\otimes}T\mA_M$ is a perturbation operator 
	with respect to  $D(\mF_{Z,T})$ over $K$.

    The following theorem is the Clifford module version of \cite[Lemma 
    2.16]{Liu2016}, which is related to \cite[Theorem 3.1]{MR1942300}, 
    \cite[Theorem 5.11]{MR2072502} and 
    \cite[Theorem 3.4]{MR3626562}.
        
	\begin{thm}\label{d188}
		For any compact submanifold $K$ of $B$,
		there exists $T_0>0$ such that
		 for $T\geq T_0$, modulo exact forms, over $K$, we have
		\begin{multline}\label{d189}
		\widetilde{\eta}_g(\mF_{Z,T},  
		1\widehat{\otimes}T\mA_M)=\int_{Y^g}\widehat{\mathrm{A}}_g(TY,\nabla^{TY})
		 \ch_g(\mE_Y/\mS, \nabla^{\mE_Y})\, 
		\widetilde{\eta}_g(\mF_{M},\mA_M)
		\\
		-\int_{Z^g}\left(\widetilde{\widehat{\mathrm{A}}}_g\cdot
		\widetilde{\ch_g}\right)\left(\nabla_T^{TZ}, \nabla^{TY,TM}, 
		\nabla_T^{\mE_Z}, \,^0\nabla^{\mE_Z}\right).
		\end{multline}

	\end{thm}

\section{Embedding of equivariant eta forms}\label{s02}

In this section, we state our main result and the application in 
equivariant Atiyah-Hirzebruch direct image. In Section 2.1, we describe the 
geometry of the embedding of submersions. In Section 2.2, we explain the 
equivariant family version of the fundamental assumption. In Section 2.3, we 
introduce the equivariant Atiyah-Hirzebruch direct image. 
In 
Section 2.4, we state our main result. 

\subsection{Embedding of submersions}\label{s0201}
In this subsection, we introduce the embedding of submersions, the setting 
of which is the same as \cite[Section 1]{Bismut1997} and \cite{MR2097553}.

Let $i:W\rightarrow V$ be an embedding of smooth oriented manifolds. Let 
$\pi_V:V\rightarrow B$ be a submersion of smooth oriented manifolds with 
closed fibres $X$, whose 
restriction $\pi_W:W\rightarrow B$ is a smooth submersion with closed fibres 
$Y$.

Thus, we have the diagram of maps
\begin{center}\label{bl0347}
	\begin{tikzpicture}[>=angle 90]
	\matrix(a)[matrix of math nodes,
	row sep=2em, column sep=2.5em,
	text height=1.5ex, text depth=0.25ex]
	{&Y&W\\
		&X&V&B.\\};
	\path[->](a-1-2) edge node[left]{\footnotesize{$i$}} (a-2-2);
	\path[->](a-1-2) edge (a-1-3);
	\path[->](a-2-2) edge (a-2-3);
	\path[->](a-1-3) edge node[left]{\footnotesize{$i$}} (a-2-3);
	\path[->](a-2-3) edge node[above]{\footnotesize{$\pi_V$}} (a-2-4);
	\path[->](a-1-3) edge node[above]{\footnotesize{$\pi_W$}} (a-2-4);
	\end{tikzpicture}
\end{center}

 In general, $B$, $V$, $W$ are not connected. We simply assume that $B$ and 
 $V$ are connected.
 For any connected component $W_{\alpha}$ of $W$, we assume that $\dim 
 V-\dim W_{\alpha}$ is even. To simplify the notations, we usually denote 
 the connected component by $W$ when there is no confusion.

Let $TX=TV/B$, $TY=TW/B$ be the relative tangent bundles to the fibres $X$, 
$Y$.
Let $T^HV$ be a smooth subbundle of $TV$ such that
\begin{align}\label{bl0349}
TV=T^HV\oplus TX.
\end{align}
Let $\widetilde{N}_{Y/X}$ be a smooth subbundle of $TX|_W$ such that
\begin{align}\label{bl0350}
TX|_W=TY\oplus \widetilde{N}_{Y/X}.
\end{align}
Let $N_{W/V}$ be the normal bundle to $W$ in $V$, which we usually denote by 
$N_{Y/X}$. Clearly,
\begin{align}\label{bl0351}
T^HV\simeq \pi_V^*TB, \quad \widetilde{N}_{Y/X}\simeq N_{Y/X}.
\end{align}
By (\ref{bl0349}) and (\ref{bl0350}), we get
\begin{align}\label{bl0352}
TV|_W=T^HV|_W\oplus  TY\oplus\widetilde{N}_{Y/X}.
\end{align}
By (\ref{bl0352}), there is a well-defined morphism
\begin{align}\label{bl0353}
\frac{TW}{TY}\rightarrow T^HV|_W\oplus \widetilde{N}_{Y/X}
\end{align}
and this morphism maps $TW/TY$ into a subbundle of $TW$.
Let $T^HW$ be the subbundle of $TW$ which is the image of $TW/TY$ by the 
morphism (\ref{bl0353}). Clearly,
\begin{align}\label{bl0354}
TW=T^HW\oplus TY.
\end{align}
In general, the subbundle $T^HW$ is not equal to $T^HV|_W$.

Let $g^{TV}$ be a metric on $TV$. Let $g^{TW}$ be the induced metric on 
$TW$. Let $g^{TX}$, $g^{TY}$ be the induced metrics on $TX$, $TY$. Note that 
even if $g^{TV}$ is of the type as in (\ref{e01013}), in general, $g^{TW}$ 
is not of this type.

We identity $N_{Y/X}$ with the orthogonal bundle $\widetilde{N}_{Y/X}$ to 
$TY$ in $TX|_W$ with respect to $g^{TX}|_W$. Let $g^{N_{Y/X}}$  be the 
induced metric on $N_{Y/X}$. On $W$, we have
\begin{align}\label{bl0380}
TX|_W=TY\oplus N_{Y/X}.
\end{align}

To the pairs $(T_{\pi_V}^HV,g^{TX})$ and $(T_{\pi_W}^HW, g^{TY} 
)$, we can 
associate the objects that we construct in (\ref{e01014}) and (\ref{e01017}).
In particular, $TX$, $TY$ are now equipped with connections $\nabla^{TX}$, 
$\nabla^{TY}$ which preserve the metrics $g^{TX}$, $g^{TY}$ respectively.

Let $P^{TY}$, $P^{N_{Y/X}}$ be the orthogonal projections $TX|_W\rightarrow 
TY$, $TX|_W\rightarrow N_{Y/X}$. By \cite[Theorem 1.9]{Bismut1997}, we have
\begin{align}\label{bl0367}
\nabla^{TY}=P^{TY}\nabla^{TX|_Y}.
\end{align}

Let 
\begin{align}\label{bl0967} 
\nabla^{N_{Y/X}}=P^{N_{Y/X}}\nabla^{TX}
\end{align}
be the connection on $N_{Y/X}$.
Then $\nabla^{N_{Y/X}}$ preserves the metric $g^{N_{Y/X}}$. Put
\begin{align}\label{bl0371}
\nabla^{TY, N_{Y/X}}=\nabla^{TY}\oplus \nabla^{N_{Y/X}}.
\end{align}
Then $\nabla^{TY,N_{Y/X}}$ is a Euclidean connection on $TX|_W=TY\oplus 
N_{Y/X}$.

Let $G$ be a compact Lie group. We assume that $W$, $V$ and $B$ are 
$G$-manifolds and the $G$-action commutes with the embedding and $\pi_V$.  
Obviously, the group action commutes with $\pi_W$. We assume that $G$ acts 
trivially on $B$. We assume that the group action preserve the splittings 
(\ref{bl0349}) and (\ref{bl0354}) and all metrics and connections are 
$G$-invariant.

Let $W^g$, $V^g$ be the fixed point sets of $W$, $V$ for $g\in G$. Then 
$\pi_W|_{W^g}:W^g\rightarrow B$ and  $\pi_V|_{V^g}:V^g\rightarrow B$ are 
submersions with closed fibres $Y^g$ and $X^g$. 
We assume that $TY^g$ and $TX^g$ are all oriented and the orientations are 
compatible with those of $TY$, $TX$ and the normal bundles as in the 
arguments at the beginning of Section \ref{s0103}.

\begin{rem}\label{bl0672}(cf. \cite[Section 7.5]{Bismut1997})
	Given $G$-equivariant pair $(T_{\pi_W}^HW, g^{TY})$, we could take 
	metrics $g^{TB}$ and $g^{TW}$ on $TB$ and $TW$ such that 
	$g^{TW}=\pi_W^*g^{TB}\oplus g^{TY}$. Let $g^N$ be a $G$-invariant metric 
	on $N_{Y/X}$. Let $\nabla^N$ be a $G$-invariant Euclidean connection on 
	$(N_{Y/X}, g^N)$ and $T^HN$ be the horizontal subbundle associated with 
	the fibration $\pi_N:N_{Y/X}\rightarrow W$ and $\nabla^N$. We take 
	$g^{TN}=\pi_N^*g^{TW}\oplus g^{N}$ for $TN=T^HN\oplus N$. 
	Since $W$ intersects $X$ orthogonally, we could take horizontal 
	subbundle $T_{\pi_V}^HV$ over $V$ such that 
	$T_{\pi_V}^HV|_W=T_{\pi_W}^HW$.
	By using the partition of unity argument, we could construct 
	$G$-invariant metrics $g^{TX}$, $g^{TV}$ on $TX$, $TV$ such that 
	$g^{TV}=\pi_V^*g^{TB}\oplus g^{TX}$ and $W$ is a 
	totally geodesic submanifold of $V$. In this case, for any $b\in B$, the 
	fibre $Y_b$ is a totally geodesic 
	submanifold of $X_b$.
	It means that
	$
	\nabla^{TX|_W}=\nabla^{TY,N_{Y/X}}.
	$
\end{rem}

By Remark \ref{bl0672},
in this paper, we will always assume that the pairs $(T_{\pi_W}^HW, 
g^{TY})$ and $(T_{\pi_V}^HV, g^{TX})$ satisfy the conditions that 
\begin{align}\label{bl0960} 
T_{\pi_V}^HV|_W=T_{\pi_W}^HW,\quad \nabla^{TX|_W}=\nabla^{TY,N_{Y/X}}.
\end{align}

\subsection{Embedding of the geometric families}\label{s0202}
In this subsection, we state our assumptions on the embedding of the 
geometric 
families, which is the equivariant family case of the assumptions in 
\cite[Section 1 b)]{MR1214954}.

Let $\mF_{Y}:=(W, \mE_Y, T_{\pi_W}^HW, g^{TY}, h^{\mE_Y}, \nabla^{\mE_Y})$ 
and $\mF_{X}:=(V, \mE_{X}, T_{\pi_V}^HV, g^{TX}, h^{\mE_{X}}, 
\nabla^{\mE_{X}})$ 
be two equivariant geometric families over $B$ such that the pairs $( 
T_{\pi_W}^HW, 
g^{TY})$ and $(T_{\pi_V}^HV, g^{TX})$ satisfy (\ref{bl0960}).
For simplicity, we assume that $\tau^{\mE_Y/\mS}\equiv 1$ on $\mE_Y$.

\textbf{Assume that $(N_{Y/X},g^{N_{Y/X}})$ has an equivariant Spin$^c$ 
structure}. Then there exists an equivariant complex line bundle $L_N$  
(cf. \cite[Appendix 
D]{MR1031992}) 
such that $w_2(N_{Y/X})=c_1(L_N)\,\mathrm{mod}\,2$, where $w_2$ is the 
second 
Stiefel-Whitney class and $c_1$ is the first Chern class.  Let 
$\mS(N_{Y/X},L_N)$ be the spinor 
bundle for $L_N$ which locally may be written as
\begin{align}\label{local206} 
\mS(N_{Y/X},L_N)=\mS_0(N_{Y/X})\otimes L_N^{1/2},
\end{align}
where $\mS_0(N_{Y/X})$ is the spinor bundle for the (possibly 
non-existent) spin structure on $N_{Y/X}$ and $L_N^{1/2}$ is the (possibly 
non-existent) square root of $L_N$. 
Then the 
$G$-action on $N_{Y/X}$ and $L_N$ lift to $\mS(N_{Y/X},L_N)$.
For simplicity, 
We usually simply denote the spinor bundle by $\mS_N$.

Let $h^L$ be a 
$G$-invariant 
Hermitian metric on $L_N$. Let $\nabla^L$ be a $G$-invariant  Hermitian 
connection on 
($L_N,h^L$).
Let $h^{\mS_N}$ be the equivariant Hermitian metric on 
$\mS_N$ induced by $g^{N_{Y/X}}$ and $h^L$. Let $\nabla^{\mS_N}$ 
be 
the equivariant Hermitian connection on 
$\mS_N$ induced by $\nabla^{N_{Y/X}}$ and $\nabla^L$.
%

From (\ref{bl0927}),
the bundle $\End_{C(TX)}(\mE_X)$ is naturally 
$\Z_2$-graded with respect to $\tau^{\mE_X/S}$.
Let $\mathcal{V}$ be a smooth self-adjoint section of $\End_{C(TX)}(\mE_X)$ 
such that it exchanges this $\Z_2$-grading and 
commutes with the $G$-action.
Then $\mV$ could be extended on 
$\pi_V^*\Lambda(T^*B)\widehat{\otimes}\mE_X$ in the same way as the 
perturbation operator $\mA$ in (\ref{bl0958}).

\textbf{We assume that on $V\backslash W$, $\mathcal{V}$ is invertible, and 
that on 
$W$, $\ker \mathcal{V}$ has locally constant nonzero dimension, so that 
$\ker \mathcal{V}$ is a nonzero smooth $\Z_2$-graded $G$-equivariant vector 
subbundle of $\mE_X|_{W}$.} Let $h^{\ker\mathcal{V}}$ be the metric on 
$\ker\mathcal{V}$ induced by the metric $h^{\mE_X|_W}$. Let 
$P^{\ker\mathcal{V}}$ be the orthogonal projection operator from 
$\mE_X|_{W}$ to $\ker \mathcal{V}$.

For $y\in W$, $U\in T_yX$, let $\partial_U\mathcal{V}(y)$ be the derivative 
of $\mathcal{V}$ with respect to $U$ in any given smooth trivialization of 
$\mE_X$ near $y\in W$. One then verifies that 
$P^{\ker\mathcal{V}}\partial_U\mathcal{V}(y)P^{\ker\mathcal{V}}$ does not 
depend on the trivialization, and only depends on the image $Z$ of $U\in 
T_yX$ in $N_{Y/X}$. From now on, we will write 
$\dot{\partial}_Z(\mathcal{V})(y)$ instead of 
$P^{\ker\mathcal{V}}\partial_U\mathcal{V}(y)P^{\ker\mathcal{V}}$. Then one 
verifies that $\dot{\partial}_Z(\mathcal{V})(y)$ is a self-adjoint element 
of $\End(\ker\mathcal{V})$ and exchanges the $\Z_2$-grading.

If $Z\in N_{Y/X}$, let $\tilde{c}(Z)\in \End(\mS_N^*)$ be the transpose of 
$c(Z)$ acting on $\mS_N$.

Denote by $N_{\C}^*=N_{Y/X}^*\otimes_{\R}\C$.
Since $L_N\otimes L_N^*$ is an equivariant trivial bundle, we have 
$\Lambda(N_{\C}^*)\simeq \mS_N\widehat{\otimes}\mS_N^*$.
We equip $\Lambda(N_{\C}^*)\widehat{\otimes}\mE_Y$ with the induced 
metric $h^{\Lambda(N_{\C}^*)\widehat{\otimes}\mE_Y}$. For $Z\in 
N_{Y/X}$, $\tilde{c}(Z)$ acts on  
$\mS_N\widehat{\otimes}\mS_N^*\widehat{\otimes}\mE_Y$ like 
$1\otimes\tilde{c}(Z)\otimes 1$. 

\textbf{Fundamental assumption}:
Let $\pi_N:N_{Y/X}\rightarrow W$ be the 
projection.
Over the total space $N_{Y/X}$, we have the equivariant identification
\begin{multline}\label{bl1011} 
\left(\pi_N^*\ker\mathcal{V}, \pi_N^*h^{\ker\mathcal{V}}, 
\dot{\partial}_Z(\mathcal{V})(y) \right)
\\
\simeq \left(\pi_N^*(\Lambda(N_{\C}^*)\widehat{\otimes}\mE_{Y}), 
\pi_N^*h^{\Lambda(N_{\C}^*)\widehat{\otimes}\mE_{Y}}, 
\sqrt{-1}\tilde{c}(Z)\right).
\end{multline} 
Let $\nabla^{\ker\mathcal{V}}$ be the equivariant Hermitian connection on 
$\ker\mathcal{V}$,
\begin{align}\label{bl1012} 
\nabla^{\ker\mathcal{V}}=P^{\ker\mathcal{V}}\nabla^{\mE_X|_W}P^{\ker\mathcal{V}}.
\end{align}
We make the assumption that under the identification (\ref{bl1011}), 
\begin{align}\label{bl1013} 
\nabla^{\ker\mathcal{V}}=\nabla^{\Lambda(N_{\C}^*)\widehat{\otimes}\mE_Y}.
\end{align}

\subsection{Atiyah-Hirzebruch direct image}\label{s0203}
In this subsection, we introduce an important example of the embedding of 
equivariant geometric families satisfying the fundamental assumption: the 
equivariant version of the Atiyah-Hirzebruch direct image 
\cite{MR0110106,MR2532714}.
We assume that the base space $B$ is compact and adopt the notations and the 
assumptions in 
Section 
\ref{s0201} in this subsection. 

We further assume that $TY$ and $TX$ have equivariant Spin$^c$ structures.
 Then there exist equivariant complex line bundles $L_Y$ and $L_X$
over $W$ and $V$ such that $w_2(TY)= c_1(L_Y) \mod 2$ and 
$w_2(TX)= c_1(L_X) \mod 2$.
Then from the splitting (\ref{bl0380}), the equivariant vector bundle 
$N_{Y/X}$ over $W$ has an equivariant Spin$^c$ structure with associated 
equivariant line bundle
$
L_N:=L_X\otimes L_Y^{-1}.
$
Let $h^{L_Y}$, $h^{L_X}$ be $G$-invariant Hermitian metrics on $L_Y$, $L_X$ 
and 
$\nabla^{L_Y}$, $\nabla^{L_X}$ be $G$-invariant Hermitian connections on 
$(L_Y,h^{L_Y})$, 
$(L_X,h^{L_X})$.
Let $h^{L_N}$ and $\nabla^{L_N}$ be metric and connection on $L_N$ induced 
by $h^{L_Y}$, $h^{L_X}$ and $\nabla^{L_Y}$, $\nabla^{L_X}$.
Let $\mS(TY,L_Y)$, $\mS(TX, L_X)$ and $\mS(N_{Y/X}, L_N)$ be the spinor 
bundles for $(TY,L_Y)$, ($TX, L_X$) and ($N_{Y/X}, L_N$), which we 
will simply denote by $\mS_Y$, $\mS_X$ and $\mS_N$. Then these spinors are 
$G$-equivariant vector bundles.
Furthermore, $\mS_X|_W=\mS_Y\widehat{\otimes}\mS_N$.
 Since $\dim N_{Y/X}=\dim 
V-\dim W$ is even, the spinor $\mS_N$ is $\Z_2$-graded.

Let $\{W_{\alpha}\}_{\alpha=1,\cdots,k}$ be the connected components of $W$.
Let $(\mu, h^{\mu})$ be a $G$-equivariant Hermitian vector bundle over $W$ 
with a $G$-invariant Hermitian connection $\nabla^{\mu}$.  In the 
followings, we will describe a geometric realization of the 
Atiyah-Hirzebruch direct image $i![\mu]\in \widetilde{K}_G^0(V)$ as in 
\cite{MR0110106,MR2532714}. We denote by $\mu_{\alpha}$ the restriction of 
$\mu$ on $W_{\alpha}$. 	

For any $r>0$, set $N_{\alpha,r}:=\{Z\in N_{Y_{\alpha}/X}: 
|Z|<r\}$. Then there is $\var_0>0$ such that
the map $(y,Z)\in N_{Y_{\alpha}/X}\rightarrow \exp_y^V(Z)$ is a 
diffeomorphism of
$N_{\alpha,2\var_0}$ on an open $G$-equivariant tubular neighbourhood of 
 $W_{\alpha}$ in $V$ for any $\alpha$. Without confusion we will also regard 
 $N_{\alpha,2\var_0}$ as the open $G$-equivariant tubular neighbourhood of 
 $W$ in 
 $V$. We choose $\var_0>0$ small enough such that for any $1\leq \alpha\neq 
 \beta\leq k$, $N_{\alpha,2\var_0}\cap N_{\beta,2\var_0} =\emptyset$.

Let $\pi_{\alpha}:N_{Y_{\alpha}/X}\rightarrow W_{\alpha}$ denote the 
projection of the normal bundle $N_{Y_{\alpha}/X}$ over $W_{\alpha}$. For 
$Z\in N_{Y_{\alpha}/X}$, let $\tilde{c}(Z)\in \End(\mS_{N_{\alpha}}^*)$ be 
the transpose of $c(Z)$ acting on $\mS_{N_{\alpha}}$.  Let 
$\pi_{\alpha}^*(\mS_{N_{\alpha}}^*)$ be the pull back bundle of 
$\mS_{N_{\alpha}}^*$ over $N_{Y_{\alpha}/X}$. For any $Z\in 
N_{Y_{\alpha}/X}$ with $Z\neq 0$, $\tilde{c}(Z): 
\pi_{\alpha}^*(\mS_{N_{\alpha},\pm}^*)|_Z\rightarrow 
\pi_{\alpha}^*(\mS_{N_{\alpha},\mp}^*)|_Z$ is an equivariant isomorphisms at 
$Z$.

From the equivariant Serre-Swan theorem \cite[Proposition 2.4]{MR0234452}, 
there exists a
$G$-equivariant Hermitian vector bundle $(E_{\alpha},h^{E_{\alpha}})$ such 
that $\mS_{N_{\alpha},-}^*\otimes \mu_{\alpha}\oplus E_{\alpha}$ is a 
$G$-equivariant trivial complex vector bundle over $W_{\alpha}$.   Then
\begin{align}\label{bl0034}
\tilde{c}(Z)\oplus 
\pi_{\alpha}^*\mathrm{Id}_{E_{\alpha}}:\pi_{\alpha}^*(\mS_{N_{\alpha},+}^*
\otimes \mu_{\alpha}\oplus E_{\alpha})\rightarrow 
\pi_{\alpha}^*(\mS_{N_{\alpha},-}^*\otimes\mu_{\alpha}\oplus E_{\alpha})
\end{align}
induces a $G$-equivariant isomorphism between two equivariant trivial vector 
bundles over $N_{\alpha,2\var_0}\backslash W_{\alpha}$.

By adding the equivariant trivial bundles, we could assume that for any 
$1\leq \alpha\neq \beta\leq k$, $\dim (\mS_{N_{\alpha},\pm}^*\otimes 
\mu_{\alpha}\oplus E_{\alpha})=\dim (\mS_{N_{\beta},\pm}^*\otimes 
\mu_{\beta}\oplus E_{\beta})$.
Clearly, $\{\pi_{\alpha}^*(\mS_{N_{\alpha},\pm}^*\otimes \mu_{\alpha}\oplus 
E_{\alpha})|_{\partial N_{\alpha,2\var_0}}\}_{\alpha=1,\cdots,k}$ extend 
smoothly to two equivariant trivial complex vector bundles over $V\backslash 
\cup_{1\leq\alpha\leq k}N_{\alpha,2\var_0}$. 

In summary, what we get is a 
$\Z_2$-graded Hermitian vector bundle ($\xi, h^{\xi}$) such that
\begin{align}\label{bl0035}
\begin{split}
&\xi_{\pm}|_{N_{\alpha,\var_0}}=\pi_{\alpha}^*(\mS_{N_{\alpha},\pm}^*\otimes 
\mu_{\alpha}\oplus E_{\alpha})|_{ N_{\alpha,\var_0}},
\\ 
&h^{\xi_{\pm}}|_{N_{\alpha,\var_0}}=\left.\pi_{\alpha}^*\left(h^{\mS_{N_{\alpha},\pm}^*\otimes
 \mu_{\alpha}}\oplus h^{E_{\alpha}}\right)\right|_{N_{\alpha,\var_0}},
\end{split}
\end{align}
where $h^{\mS_{N_{\alpha},\pm}^*\otimes \mu_{\alpha}}$ is the equivariant 
Hermitian metric on $\mS_{N_{\alpha},\pm}^*\otimes \mu_{\alpha}$  induced 
by  $g^{N_{\alpha}}$, $h^{L_{N_{\alpha}}}$ and $h^{\mu_{\alpha}}$.
Let 
$\nabla^{E_{\alpha}}$ be a $G$-invariant Hermitian connection on 
$(E_{\alpha}, h^{E_{\alpha}})$.
We can also get a $G$-invariant $\Z_2$-graded Hermitian connection 
$\nabla^{\xi}=\nabla^{\xi_+}\oplus \nabla^{\xi_-}$ on $\xi=\xi_+\oplus 
\xi_-$ over $V$ such that
\begin{align}\label{bl0037}
\nabla^{\xi_{\pm}}|_{N_{\alpha,\var_0}}=\pi_{\alpha}^*\left(\nabla^{\mS_{N_{\alpha},\pm}^*\otimes
	\mu_{\alpha}}\oplus \nabla^{E_{\alpha}}\right),
\end{align}
where $\nabla^{\mS_{N_{\alpha},\pm}^*\otimes \mu_{\alpha}}$ is the 
equivariant Hermitian connection on $\mS_{N_{\alpha},\pm}^*\otimes 
\mu_{\alpha}$  induced by  $\nabla^{N_{\alpha}}$, $\nabla^{L_{N_{\alpha}}}$ 
and $\nabla^{\mu_{\alpha}}$.

It is easy to see that there exists an equivariant self-adjoint automorphism 
$\mathcal{V}$ of $\mS_X\widehat{\otimes}\xi$, which exchanges the 
$\Z_2$-grading of $\xi$, such that
\begin{align}\label{bl0036}
\mathcal{V}|_{N_{\alpha,\var_0}}=\Id_{\mS_X} \widehat{\otimes} 
\left(\sqrt{-1}\,\tilde{c}(Z)\oplus \pi^*\mathrm{Id}_{E_{\alpha}}\right).
\end{align}

From the construction above, we could see that $\mathcal{V}$ is invertible 
on $V\backslash W$ and
\begin{align}\label{bl0038}
(\ker \mathcal{V})|_{W}=\mS_X|_{W}\widehat{\otimes} \mS^*_N\otimes 
\mu=\mS_Y\widehat{\otimes}\mS_N\widehat{\otimes}\mS^*_N\otimes 
\mu=\mS_Y\widehat{\otimes} \Lambda(N_{\C}^*)\otimes \mu
\end{align}
is an equivariant vector bundle over $W$. Let $P^{\ker \mathcal{V}}$ be the 
orthogonal projection from $\mS_X\widehat{\otimes}\xi|_{W}$ onto $\ker 
\mathcal{V}$ and
$
\nabla^{\ker \mathcal{V}} =P^{\ker 
\mathcal{V}}\nabla^{\mS_X\widehat{\otimes}\xi|_{W}}P^{\ker \mathcal{V}}.
$
From (\ref{bl0960}), we have
\begin{align}\label{bl0040}
\nabla^{\ker \mathcal{V}}=\nabla^{\mS_Y\widehat{\otimes} 
\Lambda(N_{\C}^*)\otimes \mu}.
\end{align}

Here $[\xi_+]-[\xi_-]\in \widetilde{K}_G^0(V)$ is an equivariant version of 
the Atiyah-Hirzebruch direct image $i![\mu]$ in \cite{MR0110106}.
In this construction, let $\mE_Y=\mS_Y\otimes\mu$ and 
$\mE_{X,\pm}=\mS_X\widehat{\otimes}\xi_{\pm}$. Then it satisfies all 
assumptions in Section \ref{s0202}.

\subsection{Main result}
In this subsection, we state our main result.

 Let $\mF_Y$ and $\mF_X$ be the equivariant geometric families satisfying 
the assumptions in 
Section \ref{s0202}.

For $T\geq 0$, let $\nabla^{\mE_X,T}$ be the superconnection on
$\mE_X$ by
\begin{align}\label{bl0381}
\nabla^{\mE_X,T}=\nabla^{\mE_X}+\sqrt{T}\mathcal{V}.
\end{align}
Let $R_T^{\mE_X/\mS}$ be the twisting curvature of $\nabla^{\mE_X,T}$ as in 
(\ref{e01038}). Let $\dim(N_{X^g/X})=\ell_2$.
For $T>0$, by \cite{MR790678} and (\ref{bl1009}), we have the equivariant 
version of 
\cite[(1.17)]{MR1214954}:
\begin{multline}\label{bl0546}
\frac{\partial}{\partial 
	T}\tr^{\mE_X/\mS}\left[\sigma_{\ell_2}(g^{\mE_X})\exp\left(-R_T^{\mE_X/\mS}|_{V^g}\right)\right]
\\
=-d\tr^{\mE_X/\mS}\left[\sigma_{\ell_2}(g^{\mE_X})\frac{\mathcal{V}|_{V^g}}{2\sqrt{T}}\exp\left(-R_T^{\mE_X/\mS}|_{V^g}\right)\right].
\end{multline}

Recall that $\psi$ is the operator defined in (\ref{e01059}). The proof of 
the following theorem is the same as those of \cite[Theorem 6.3]{MR1316553} 
and \cite[Theorem 1.2]{MR1214954}.

\begin{thm}\label{bl0382}
	For any compact set $K\subset V^g$, there exists $C>0$, such that if 
	$\omega\in \Omega^*(V^g)$ has support in $K$,
	\begin{multline}\label{bl0383}
	\left|\int_{X^g}\omega \cdot 
	\frac{2^{\ell_2/2}}{\mathrm{det}^{1/2}(1-g|_{N_{X^g/X}})}\psi_{V^g} 
	\tr^{\mE_X/\mS}\left[\sigma_{\ell_2}(g^{\mE_X})
	\exp\left(-R_T^{\mE_X/\mS}|_{V^g}\right)\right]\right.
	\\
	\left.-\int_{Y^g}\omega\cdot\widehat{\mathrm{A}}_g^{-1}(N_{Y/X},
	\nabla^{N_{Y/X}})\ch_g(\mE_Y/\mS,
	\nabla^{\mE_Y})\right|
	\leq \frac{C}{\sqrt{T}}\|\omega\|_{\mathscr{C}^1(K)},
	\end{multline}
	and
	\begin{align}\label{bl0545}
	\left|\int_{X^g}\omega\cdot \psi_{V^g} 
	\tr^{\mE_X/\mS}\left[\sigma_{\ell_2}(g^{\mE_X})\frac{\mathcal{V}|_{V^g}}{2\sqrt{T}}\exp\left(-R_T^{\mE_X/\mS}|_{V^g}\right)\right]\right|
	\leq \frac{C}{T^{3/2}}\|\omega\|_{\mathscr{C}^1(K)}.
	\end{align}
\end{thm}

Now we could extend the Bismut-Zhang current in \cite[Definition 
1.3]{MR1214954} to the equivariant case.
\begin{defn}\label{bl0554}
	The equivariant Bismut-Zhang current $\gamma_g^X(\mF_Y,\mF_X)$ over 
	$V^g$ is defined by
	\begin{multline}\label{bl0555}
	\gamma_g^X(\mF_Y,\mF_X)=\frac{1}{2\sqrt{\pi 
		}\sqrt{-1}}\cdot\frac{2^{\ell_2/2}}{\mathrm{det}^{1/2}(1-g|_{N_{X^g/X}})}
	\\
	\cdot\int_{0}^{\infty}\psi_{V^g} 
	\tr^{\mE_X/\mS}\left[\sigma_{\ell_2}(g^{\mE_X})\mathcal{V}|_{V^g}\exp\left(-R_T^{\mE_X/\mS}|_{V^g}\right)\right]\frac{dT}{2\sqrt{T}}.
	\end{multline}
\end{defn}

By Theorem \ref{bl0382}, the current $\gamma_g^X(\mF_Y,\mF_X)$ is 
well-defined.

Let $\delta_{W^g}$ be the current of integration over the submanifold $W^g$ 
in $V^g$. By integrating (\ref{bl0546}) and using Theorem \ref{bl0382}, we 
have the following equivariant extension of \cite[Theorem 1.4]{MR1214954}.

\begin{thm}\label{bl0558}
	The following equation of currents holds
	\begin{multline}\label{bl0559}
	d\gamma_g^X(\mF_Y,\mF_X)=\ch_g(\mE_X/\mS, \nabla^{\mE_X})
	\\
	-\widehat{\mathrm{A}}_g^{-1}(N_{Y/X},\nabla^{N_{Y/X}})\ch_g(\mE_Y/\mS, 
	\nabla^{\mE_Y})\delta_{W^g}.
	\end{multline}
\end{thm}

\begin{rem}\label{bl0561}
	Similarly as in \cite{MR1214954}, the wave front set 
	$\mathrm{WF}(\gamma_g^X)$ of the current $\gamma_g^X$ is included in 
	$N_{W^g/V^g}^*$ and $\gamma_g^X(\mF_Y,\mF_X)$ is a locally integrable 
	current.
\end{rem}

\begin{prop}\label{bl0676}
	Let $\mathcal{A}_Y$ be a perturbation operator with respect to 
	$D(\mF_Y)$.
	Then there exists a family of 
	bounded pseudodifferential operator $\mA_{T,Y}$ on $\mF_X$, depending
	continuously on $T\geq 1$, such that
	the norm of $\mA_{T,Y}$ is the same as that of $\mA_Y$ for any $T\geq 1$ 
	and
	 for any compact submanifold $K$ of $B$, 
	there exists $T_0\geq 1$ such that $T\mathcal{V}+\mA_{T,Y}$ is the 
	perturbation 
	operator 
	with respect to $D(\mF_X)$ over $K$ for $T\geq T_0$.
\end{prop}
\begin{proof} 
	Following the arguments in \cite[Section 8, 9]{MR1188532} and 
	\cite[Section 
	4b)]{MR1214954} word by word, we could construct a smooth family of 
	equivariant isometric 
	embeddings 
	\begin{align}\label{bl0983} 
	J_{T,b}:L^2(Y_b,\mE_Y|_{Y_b})\rightarrow L^2(X_b,\mE_X|_{X_b})
	\end{align}
	for $b\in B$, as in \cite[Definition 9.12]{Bismut1997}.
	
	Let $\mathbb{E}_{T,b}$ be the image of $L^2(Y_b,\mE_Y|_{Y_b})$ in 
	$L^2(X_b,\mE_X|_{X_b})$ by 
	$J_{T,b}$. Let $\mathbb{E}_{T,b}^{\bot}$ be the orthogonal space to 
	$\mathbb{E}_{T,b}$ 
	in $L^2(X_b,\mE_X|_{X_b})$. Since $J_{T,b}$ is an isometric embedding, 
	$J_{T,b}:L^2(Y_b,\mE_Y|_{Y_b})\rightarrow \mathbb{E}_T$ is invertible. 
	We 
	extend the 
	domain of $J_{T,b}^{-1}$ to $L^2(X_b,\mE_X|_{X_b})$ such that it 
	vanishes on 
	$\mathbb{E}_{T,b}^{\bot}$.
	
	Let $\mA_{T,Y}=\{\mA_{T,Y,b} \}_{b\in B}$ be the family of bounded 
	pseudodifferential operators
	\begin{align}\label{bl0232}
	\mA_{T,Y,b}:=J_{T,b} \mA_{Y,b} 
	J_{T,b}^{-1}:L^2(X_b,\mE_X|_{X_b})\rightarrow L^2(X_b,\mE_X|_{X_b}).
	\end{align}
	Then $\mA_{T,Y}$ is a smooth family of equivariant self-adjoint 
	operators. From the definition of the perturbation 
	operator $\mA_Y$, we see that $\mA_{T,Y}$ commutes (resp. anti-commutes) 
	with the $\Z_2$-grading $\tau^{\mE}$ of 
	$\mE_X$ when the fibres are odd (resp. even) dimensional. Since $J_T$ is 
	isometric, the $L^2$-norm of $\mA_{T,Y}$ is the same as that of $\mA_Y$.

	Since $J_T$ is continuous with respect to $T$, so is the operator 
	$\mA_{T,Y}$. We only need to prove that $D(\mF_X)+T\mathcal{V}+ 
	\mA_{T,Y}$ 
	over $K$ is invertible for $T$ large enough. 
	
	Over a compact submanifold $K$ of $B$, the same estimates of 
	$D(\mF_X)+T\mathcal{V}$ as \cite[Theorem 9.8, 9.10, 9.11]{MR1188532} 
	hold. 
	Since $D(\mF_Y)+\mA_Y$ is invertible, the arguments in \cite[Section 
	9]{MR1188532}, in which we replace $D(\mF_Y)$ and 
	$D(\mF_X)+T\mathcal{V}$ by 
	$D(\mF_Y)+\mA_Y$ and $D(\mF_X)+T\mathcal{V}+\mA_{T,Y}$, imply that 
	there exists $T_0\geq 1$, depending on $K$, such that for any $T\geq 
	T_0$,
	$D(\mF_X)+T\mathcal{V}+\mA_{T,Y}$ is invertible. Moreover, the 
	absolutely value of the spectrum of $D(\mF_X)+T\mathcal{V}+\mA_{T,Y}$ 
	has a uniformly positive lower bound for $T\geq T_0$. 
	
	The proof of our proposition is completed.
\end{proof}

Now we state our main result of this paper.
\begin{thm}\label{bl0563}
	Let $\mA_Y$ and $\mA_{X}$ be the perturbation operators with respect to 
	$D(\mF_{Y})$ and $D(\mF_X)$. 
	Let 
	$\mA_{T,Y}$ be the operator in Proposition \ref{bl0676} with respect to 
	$D(\mF_X)$.
	Then for any compact submanifold $K$ of $B$, there 
	exists $T_0>2$ such 
	that
	 for any $T\geq T_0$, modulo exact forms, over $K$, we have
	\begin{multline}\label{bl0564}
	\tilde{\eta}_g(\mF_X, \mathcal{A}_X)=\tilde{\eta}_g(\mF_Y, 
	\mathcal{A}_Y)+\int_{X^g}\widehat{\mathrm{A}}_g(TX, 
	\nabla^{TX})\,\gamma_g^X(\mF_Y,\mF_X)
%
	\\
	+\ch_g(\mathrm{sf}_G\{D(\mF_X)+\mA_X, D(\mF_{X}) 
	+T\mathcal{V}+\mA_{T,Y}\})
	.
	\end{multline}
	
\end{thm}

Observe that 
since we only need to prove (\ref{bl0564}) over a compact submanifold, in 
the proof of Theorem \ref{bl0563}, we may assume that
$B$ is 
compact.

Assume that the base space is a point and $TY$, $TX$ have equivariant 
Spin structure. Then there exist equivariant complex vector bundles 
$\mu$ and 
$\xi_{\pm}$ such that $\mE_Y=\mS_Y\otimes\mu$ and 
$\mE_{X,\pm}=\mS_X\widehat{\otimes}\xi_{\pm}$. The following corollary is a 
direct consequence of Theorem \ref{bl0563}.

\begin{cor}\label{bl1002} 
	There exists $x\in R(G)$, the representation ring of $G$, 
	such that
	\begin{multline}\label{bl1003}
	\ov{\eta}_g(X, \xi_+)-\ov{\eta}_g(X, \xi_-)=\ov{\eta}_g(Y, 
	\mu)
	\\
	+\int_{X^g}\widehat{\mathrm{A}}_g(TX, 
	\nabla^{TX})\,\gamma_g^X(\mF_Y,\mF_X)+\chi_g(x).
	\end{multline}
	Here $x$ could be written as an equivariant spectral flow, 
	$\chi_g(x)$ is the character of $g$ on $x$ and $\ov{\eta}_g$ is the 
	equivariant reduced eta invariant.
\end{cor} 

When $g=1$, Corollary \ref{bl1002} is the modification of the Bismut-Zhang 
embedding formula by expressing the $\mathrm{mod}\,\Z$ term as a spectral 
flow. Note that in \cite[Theorem 4.1]{MR1870658}, the authors give an index 
interpretation of the $\mathrm{mod}\,\Z$ term of the embedding formula when 
the manifolds are the boundaries. It is also interesting to find the 
equivariant family extension of that formula. 


\begin{cor}
Let $X$ be an odd-dimensional closed $G$-equivariant Spin$^c$ manifold.
For $g\in G$, let $(\mu, h^{\mu})$ be an 
equivariant Hermitian vector bundle over $X^g$ with a $G$-invariant 
Hermitian connection $\nabla^{\mu}$. Then there exist a 
$\Z_2$-graded equivariant Hermitian vector bundle $(\xi,h^{\xi})$ over $X$ 
with a $G$-invariant Hermitian connection $\nabla^{\xi}$ and 
 $x\in R(G)$, such that
\begin{align}\label{bl1001}
\ov{\eta}_g(X,\xi_+)-\ov{\eta}_g(X,\xi_-)=\ov{\eta}_g(X^g, \mu)
+\chi_g(x).
\end{align}
\end{cor}
\begin{proof}
	Note that $X^g$ is naturally totally geodesic in $X$. 	Take $(\xi, 
	h^{\xi}, \nabla^{\xi})$ as 
	the equivariant Atiyah-Hirzebruch direct image of $(\mu, h^{\mu}, 
	\nabla^{\mu})$ as in Section 2.3.
	We only need to notice that $\mV|_{X^g}=0$ in this case. It implies that 
	$\gamma_g^X(\mF_{X^g},\mF_X)=0$.
\end{proof}

\begin{rem}
	Note that in \cite{MR2602854}, the authors establish an index theorem 
	for differential K-theory. The key analytical tool is the Bismut-Zhang 
	embedding formula of the reduced eta invariants in \cite{MR1214954}. 
	Using Corollary \ref{bl1002},  the index theorem there could be extended 
	to the equivariant case whenever the equivariant differential K-theory 
	is well-defined. Using Theorem \ref{bl0563}, we can also get the 
	compatibility of the push-forward map in equivariant differential 
	K-theory 
	along the proper submersion and the embedding under the model of 
	Bunck-Schick \cite{Bunke2009,MR3182258,Liu2016}. We will study these in 
	the subsequent paper.
\end{rem}

\section{Proof of Theorem \ref{bl0563}}

In this section, we prove our
main result Theorem \ref{bl0563}. In Section 3.1, we prove Theorem 
\ref{bl0563} 
when the base space is a point using some intermediary 
results along the lines of \cite{MR1214954}, the proof of which rely on 
almost 
identical arguments of \cite{MR1316553,MR1214954}. In Sections 3.2, we 
explain how to use the functoriality to reduce Theorem \ref{bl0563} to the 
case in Section 3.1. 

\subsection{Embedding of equivariant eta invariants}

 In this 
subsection, we will prove our main result when $B$ is a point and $\dim X$ 
is odd. Recall that in (\ref{bl0960}), we already assume that $Y$ is totally 
geodesic in $X$.

\begin{thm}\label{bl0981} 
	Assume that $B$ is a point and $\dim X$ is odd. 
Then there exists $T_0>2$ such that
for any $T\geq T_0$, we have 
\begin{multline}\label{bl0982}
\tilde{\eta}_g(\mF_X, \mA_{X})=\tilde{\eta}_g(\mF_Y, 
\mathcal{A}_Y)+\int_{X^g}\widehat{\mathrm{A}}_g(TX, 
\nabla^{TX})\,\gamma_g^X(\mF_Y,\mF_X)
\\
+\ch_g(\mathrm{sf}_G\{D(\mF_X)+\mA_X, D(\mF_{X}) 
+T\mathcal{V}+\mA_{T,Y}\})
.
\end{multline}
\end{thm}

Set
\begin{align}\label{bl0571}
D_{u,T}=\sqrt{u}(D(\mF_X)+T\mathcal{V}+\chi(\sqrt{u})((1-\chi(T))\mA_X
+\chi(T)\mathcal{A}_{T,Y})),
\end{align}
where $\chi$ is the cut-off function defined in (\ref{d013}).
Let
\begin{align}\label{bl0572}
\mB_{u^2,T}=D_{u^2,T}+dT\wedge\frac{\partial}{\partial 
T}+du\wedge\frac{\partial}{\partial u}.
\end{align}

\begin{defn}\label{bl0573}
	We define $\beta_g=du\wedge \beta_g^u+dT\wedge \beta_g^T$ to be the part 
	of $\pi^{-1/2}\tr_s[g\exp(-\mB_{u^2,T}^2)]$
	of degree one with respect to the coordinates $(T,u)$, with functions 
	$\beta_g^u$, $\beta_g^T: \mathbb{R}_{+,T}\times\mathbb{R}_{+,u}
	\rightarrow \R$.
\end{defn}
From (\ref{bl0572}), we have
\begin{align}\label{bl0574}
\begin{split}
&\beta_g^u(T,u)=
-\frac{1}{\sqrt{\pi}}\tr_s\left[g\frac{\partial D_{u^2,T}}{\partial 
u}\exp(-D_{u^2,T}^2)\right],
\\
&\beta_g^T(T,u)=
-\frac{1}{\sqrt{\pi}} \tr_s\left[g\frac{\partial D_{u^2,T}}{\partial 
T}\exp(-D_{u^2,T}^2)\right].
\end{split}
\end{align}
When $0<u<1$, $\chi(u)=0$. In this case,
\begin{align}\label{bl0700}
\begin{split}
&\beta_g^u(T,u)=
-\frac{1}{\sqrt{\pi}}\tr_s\left[g(D(\mF_X)+T\mathcal{V})\exp(-u^2(D(\mF_X)+T\mathcal{V})^2)\right],
\\
&\beta_g^T(T,u)=
-\frac{u}{\sqrt{\pi}} 
\tr_s\left[g\mathcal{V}\exp(-u^2(D(\mF_X)+T\mathcal{V})^2)\right].
\end{split}
\end{align}
From (\ref{i19}), 
\begin{align}\label{bl0575}
\tilde{\eta}_g(\mF_X, \mA_X)=-\int_0^{+\infty}\beta_g^u(0,u)du.
\end{align}

As in \cite[Theorem 3.4]{MR1214954} (see also \cite[Proposition 
4.2]{MR3626562}),
we have
\begin{align}\label{bl0577}
\left(du\wedge\frac{\partial}{\partial u}+dT\wedge\frac{\partial}{\partial 
T}\right)\beta_g=0.
\end{align}

Let $T_0$ be the constant in Proposition \ref{bl0676}.
Take $\var, A, T_1$, $0<\var< 1\leq A<\infty$, $T_0\leq T_1<\infty$. Let 
$\Gamma=\Gamma_{\var,A,T_1}$ be the oriented contour in
$\mathbb{R}_{+,T}\times\mathbb{R}_{+,u}$.

\

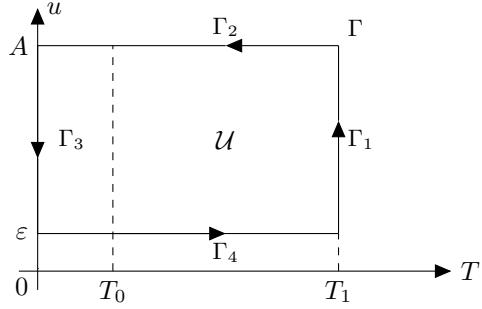
\begin{center}\label{bl0580}
	\begin{tikzpicture}
	\draw[->][ -triangle 45] (-0.25,0) -- (5.5,0);
	\draw[->][ -triangle 45] (0,-0.25) -- (0,3.5);
	\draw[->][ -triangle 45] (1,0.5) -- (2.5,0.5);
	\draw (0,0.5) -- (4,0.5);
	\draw[->][ -triangle 45] (0,3) -- (0,1.5);
	\draw (0,1.5) -- (0,0.5);
	\draw[->][ -triangle 45] (4,0.5) -- (4,2);
	\draw (4,2) -- (4,3);
	\draw[->][ -triangle 45] (4,3) -- (2.5,3);
	\draw (2.5,3) -- (0,3);
	\draw[dashed] (1,0) -- (1,3);
	\draw[dashed] (4,0) -- (4,0.5);
	\foreach \x in {0}
	\draw (\x cm,1pt) -- (\x cm,1pt) node[anchor=north east] {$\x$};
	\draw
	(2.5,1.75)  node {$\mU$}(2.5,1.75);
	\draw
	(0,3.5)  node[anchor=west] {$u$}(0,3.5);
	\draw
	(5.5,0)  node[anchor=west] {$T$}(5.5,0);
	\draw
	(0,0.5)  node[anchor=east] {$\varepsilon$}(0,0.5);
	\draw
	(0,3)  node[anchor=east] {$A$}(0,3);
	\draw
	(1,0)  node[anchor=north] {$T_0$}(1,0);
	\draw
	(4,0)  node[anchor=north] {$T_1$}(4,0);
	\draw
	(4,1.75)  node[anchor=west] {\small{$\Gamma_1$}}(4,1.75);
	\draw
	(2.5,0.5)  node[anchor=north] {\small{$\Gamma_4$}}(2.5,0.5);
	\draw
	(2.5,3)  node[anchor=south] {\small{$\Gamma_2$}}(2.5,3);
	\draw
	(0.75,1.75)  node[anchor=east] {\small{$\Gamma_3$}}(0.75,1.75);
	\draw
	(4,3)  node[anchor=south west] {\small{$\Gamma$}}(4,3);
	\end{tikzpicture}
\end{center}
\

The contour $\Gamma$ is made of four oriented pieces 
$\Gamma_1,\cdots,\Gamma_4$ indicated in the above picture.
For $1\leq k\leq 4$,
set $I_k^0=\int_{\Gamma_k}\beta_g$. Then by Stocks' formula and  
(\ref{bl0577}),
\begin{align}\label{bl0581}
\sum_{k=1}^4I_k^0=\int_{\partial 
\mU}\beta_g=\int_{\mU}\left(du\wedge\frac{\partial}{\partial u}
+dT\wedge\frac{\partial}{\partial T}\right)\beta_g=0.
\end{align}

For any $g\in G$, set
\begin{align}\label{bl0582}
\beta_g^Y(u)=\frac{1}{\sqrt{\pi}}\tr\left[g\exp\left(-\left(u(D(\mF_Y)+\chi(u)\mathcal{A}_Y)+du\wedge
 \frac{\partial}{\partial u}\right)^2\right)\right]^{du}.
\end{align}
Then by Definition \ref{d017},
\begin{align}\label{bl0583}
\tilde{\eta}_g(\mF_Y, \mathcal{A}_Y)=-\int_0^{+\infty}\beta_g^Y(u)du.
\end{align}

We now establish some estimates of $\beta_g$.

\begin{thm}\label{bl0584}
	i) For any $u>0$, we have
	\begin{align}\label{bl0585}
	\lim_{T\rightarrow \infty}\beta_g^u(T,u)=\beta_g^Y(u).
	\end{align}

	ii) For $0<u_1<u_2$ fixed, there exists $C>0$ such that, for $u\in 
	[u_1,u_2]$, $T\geq 2$, we have
	\begin{align}\label{bl0586}
	|\beta_g^u(T,u)|\leq C.
	\end{align}
	
	iii) We have the following identity:
	\begin{align}\label{bl0587}
	\lim_{T\rightarrow +\infty} 
	\int_{2}^{\infty}\beta_g^u(T,u)du=\int_{2}^{\infty}\beta_g^Y(u)du.
	\end{align}
\end{thm}
\begin{proof}
If $P$ is an operator, let $\mathrm{Spec}(P)$ be the spectrum of $P$. 
From the proof of Proposition \ref{bl0676}, 
there exist $T_0\geq 1$, $c>0$, such that	
	for $T\geq T_0$,
	\begin{align}\label{bl0314}
	\mathrm{Spec}(D(\mF_X)+T\mathcal{V}+ \mA_{T,Y})\cap [-c,c] =\emptyset.
	\end{align}

Recall that $\mathbb{E}_T^0$ is the image of $J_T$ defined in (\ref{bl0983}).
For $\delta\in[0,1]$, we write $D(\mF_X)+T\mathcal{V}+\delta \mA_{T,Y}$ in 
matrix form with respect to the splitting by
$\mathbb{E}_T^0\oplus \mathbb{E}_T^{0,\bot}$,
\begin{align}\label{bl0233}
D(\mF_X)+T\mathcal{V}+\delta \mA_{T,Y}=\left(
\begin{array}{cc}
A_{T,1}+\delta \mA_{T,Y} & A_{T,2} \\
A_{T,3} & A_{T,4} \\
\end{array}
\right).
\end{align}
By \cite[Theorem 9.8]{MR1188532} and (\ref{bl0232}),
as $T\rightarrow +\infty$, we have
\begin{align}\label{bl0234}
J_T^{-1}(A_{T,1}+\delta \mA_{T,Y})J_T=D(\mF_Y)+\delta 
\mA_Y+O\left(\frac{1}{\sqrt{T}}\right).
\end{align}

Set
\begin{align}
\mathcal{T}:=\{\delta\in[0,1]: D(\mF_Y)+\delta\mA_Y\ \text{is not 
invertible} \}.
\end{align}
Then $\mathcal{T}$ is a closed subset of $[0,1]$.

We firstly assume that $\mT$ is not empty.
Fix $\delta_0\in\mathcal{T}$. There exists $C(\delta_0)>0$ such that 
\begin{align}
\mathrm{Spec}(D(\mF_Y)+\delta_0\mA_Y)\cap [ 
-2C(\delta_0),2C(\delta_0)]=\{0\}.
\end{align}
Since the eigenvalues are continuous with respect to $\delta$, there exists 
$\var>0$ small enough, such that when $\delta\in 
(\delta_0-\var,\delta_0+\var)$, 
\begin{align}
\mathrm{Spec}(D(\mF_Y)+\delta\mA_Y)\cap [-C(\delta_0),C(\delta_0)]\subset 
(-C(\delta_0)/4,C(\delta_0)/4)
\end{align} 
and
\begin{multline}
\mathrm{Spec}(D(\mF_Y)+\delta\mA_Y)\cap (-\infty, -C(\delta_0)]\cup 
[C(\delta_0),+\infty)
\\
\subset (-\infty, -7C(\delta_0)/4)\cup (7C(\delta_0)/4,+\infty).
\end{multline} 
Then following the same process in \cite[Section 9]{MR1188532} and 
\cite[Section 4 b)]{MR1214954} by replacing 
$D(\mF_Y)$ and $D(\mF_X)+T\mathcal{V}$ by $D(\mF_Y)+\delta\mA_Y$ and 
$D(\mF_X)+T\mathcal{V}+\delta \mA_{T,Y}$, for $\alpha>0$ fixed, when $T$ is 
large enough, there exists $C>0$, such that for any $\delta\in 
(\delta_0-\var,\delta_0+\var)$, 
\begin{align}\label{bl1016} 
\begin{split}
&\left|\tr_s\left[g(D(\mF_X)+T\mathcal{V}+\delta \mA_{T,Y})\exp\left(-\alpha 
(D(\mF_X)+T\mathcal{V}+\delta \mA_{T,Y})^2 \right) \right]\right.
\\
&\quad\quad\quad\quad\quad-\tr\left.\left[g(D(\mF_Y)+\delta\mA_Y)\exp\left(-\alpha
 (D(\mF_Y)+\delta\mA_Y)^2 \right) \right] \right|\leq \frac{C}{\sqrt{T}},
\\
&\left|\tr_s\left[g\mA_{T,Y}\exp\left(-\alpha (D(\mF_X)+T\mathcal{V}+\delta 
\mA_{T,Y})^2 \right) \right]\right.
\\
&\quad\quad\quad\quad\quad-\tr\left.\left[g\mA_Y\exp\left(-\alpha
 (D(\mF_Y)+\delta\mA_Y)^2 \right) \right] \right|\leq \frac{C}{\sqrt{T}}.
\end{split}
\end{align}
Since $\mathcal{T}$ is compact, there exists an open neighborhood $\mU$ of 
$\mathcal{T}$ in $[0,1]$ such that (\ref{bl1016}) hold uniformly for 
$\delta\in \mU$. For $\delta\in [0,1]\backslash \mU$, there is a uniformly 
lower positive bound  of the absolute value of the spectrum of 
$D(\mF_Y)+\delta\mA_Y$. So the process of \cite[Section 9]{MR1188532} also 
works. It means that (\ref{bl1016}) hold uniformly for $\delta\in [0,1]$. 

If $\mT=\emptyset$, it means that there is a uniformly 
lower positive bound  of the absolute value of the spectrum of 
$D(\mF_Y)+\delta\mA_Y$ for $\delta\in [0,1]$. Thus (\ref{bl1016}) holds 
uniformly.

In summary, for $\alpha>0$ fixed, when $T$ is 
large enough, there exists $C>0$, such that for any $\delta\in 
[0,1]$, (\ref{bl1016}) holds. 

Therefore, from Definition \ref{bl0573}, (\ref{bl0571}), (\ref{bl0572}) and 
(\ref{bl0582}), we get the proof of Theorem \ref{bl0584} i) and ii).

For $u\geq 2$ and $T\geq T_0$, from Definition \ref{bl0573}, (\ref{bl0571}) 
 and (\ref{bl0572}), we have
\begin{multline}\label{bl0984} 
\beta_g^u(T,u)=
-\frac{1}{\sqrt{\pi}}\tr_s\left[g(D(\mF_X)+T\mathcal{V}+\mA_{T,Y})\right.
\\
\left.
\exp(-u^2(D(\mF_X)+T\mathcal{V}+\mA_{T,Y})^2\right].
\end{multline}
From \cite[Proposition 2.37]{MR2273508}, (\ref{bl0314}) and (\ref{bl0984}), 
there exists $C_T>0$, depending on $T\geq T_0$, such that for $u$ large 
enough,
\begin{align}\label{bl0985} 
|\beta_g^u(T,u)|\leq C_T\exp(-cu^2).
\end{align}
From the first inequality of (\ref{bl1016}) for $\delta=1$, we see that 
$C_T$ in (\ref{bl0985}) is uniformly bounded. Thus iii) follows from i) and 
the 
dominated convergence theorem.

The proof of our theorem is completed. 
\end{proof}

\begin{thm}\label{bl0588}
	Let $T_0$ be the constant in Proposition \ref{bl0676}. When 
	$u\rightarrow +\infty$, we have
	\begin{align}\label{bl0589}
	\lim_{u\rightarrow +\infty} \int_{0}^{T_0}\beta_g^T(T,u)dT
	=\ch_g(\mathrm{sf}_G\{D(\mF_X)+\mA_X, D(\mF_{X}) 
	+T_0\mathcal{V}+\mA_{T_0,Y}\}).
	\end{align}
	and
	\begin{align}\label{bl0590}
	\lim_{u\rightarrow +\infty} \int_{T_0}^{\infty}\beta_g^T(T,u)dT=0.
	\end{align}
\end{thm}
\begin{proof} 
Set
	\begin{align}
	D_{u,T}'=\sqrt{u}(D(\mF_X)+\chi(\sqrt{u})(T\mathcal{V}+((1-\chi(T))\mA_X+\chi(T)\mathcal{A}_{T,Y})))
	\end{align}
	and
\begin{align}
	\beta_g^T(T,u)'=
	-\frac{1}{\sqrt{\pi}}\tr_s\left[g\frac{\partial D_{u^2,T}'}{\partial 
	T}\exp(-(D_{u^2,T}')^2)\right].
\end{align}
	
	Note that when $u>2$, 
	\begin{align}
	\beta_g^T(T,u)'=\beta_g^T(T,u).
	\end{align}
	The proof of the anomaly formula Theorem \ref{e01105} (cf. \cite[Theorem 
	2.17]{Liu2016}) show that 
	\begin{multline}
	\lim_{u\rightarrow +\infty} 
	\int_{0}^{T_0}\beta_g^T(T,u)dT=\lim_{u\rightarrow +\infty} 
	\int_{0}^{T_0}\beta_g^T(T,u)'dT
	\\=\tilde{\eta}_g(\mF_X, \mA_{X})-\tilde{\eta}_g(\mF_X, 
	T_0\mathcal{V}+\mA_{T_0,Y})
	\\
	=\ch_g(\mathrm{sf}_G\{D(\mF_X)+\mA_X, D(\mF_{X}) 
	+T_0\mathcal{V}+\mA_{T_0,Y}\}). 	
	\end{multline}
	
	Since $D(\mF_{X}) +T\mathcal{V}+\mA_{T,Y}$ is invertible for $T\geq T_0$,
	the proof of (\ref{bl0590}) is the same as \cite[Theorem 2.22]{Liu2016}. 
	In fact, as in \cite[(6.8)]{MR3626562}, for $u'>0$ fixed, there exist 
	$C>0$, $T'\geq T_0$ and $\delta>0$ 
	such that for $u\geq u'$ and $T\geq T'$, we have
	\begin{align}\label{bl0986} 
	|\beta_g^T(T,u)|\leq \frac{C}{T^{1+\delta}}\exp(-cu^2).
	\end{align}

The proof of Theorem \ref{bl0588} is completed.	
\end{proof}

\begin{thm}\label{e03022}
	i) For any $u\in(0,1]$, there exist $C>0$ and $\delta>0$ such that, for 
	$T$ large enough, we have
	\begin{align}\label{e03023}
	|\beta_g^T(T,u)|\leq\frac{C}{T^{1+\delta}}.
	\end{align}
	
	ii) There exist $C>0$, $\gamma\in (0,1]$ such that for $u\in (0,1]$, 
	$0\leq T\leq u^{-1}$,
	\begin{multline}\label{bl0592}
	\left|u^{-1}\beta_g^T\left(\frac{T}{u}, u\right)+
	\frac{1}{2\sqrt{\pi}\sqrt{-1}}\cdot\frac{2^{\ell_2/2}}{\mathrm{det}^{1/2}
		(1-g|_{N_{X^g/X}})}\int_{X^g}
	\widehat{\mathrm{A}}_g(TX,\nabla^{TX})\right.
	\\
	\left.\cdot\psi_{X^g} 
	\tr^{\mE_X/\mS}\left[\sigma_{\ell_2}(g^{\mE_X})\mathcal{V}|_{X^g}\exp\left(-R_{T^2}^{\mE_X/\mS}|_{X^g}\right)\right]
	\right|
	\leq\frac{C(u(1+T))^{\gamma}}{\sup\{T,1\}}.
	\end{multline}
	
	iii) For any $T>0$,
	\begin{align}\label{bl0594}
	\lim_{u\rightarrow 0}u^{-2}\beta_g^T(T/u^2, u)
	=0.
	\end{align}
	
	iv) There exist $C>0$, $\delta\in (0,1]$ such that for $u\in (0,1]$, 
	$T\geq 1$,
	\begin{align}\label{bl0596}
	\left|u^{-2}\beta_g^T(T/u^2, u)\right|\leq \frac{C}{T^{1+\delta}}.
	\end{align}
\end{thm}
\begin{proof} 
It is easy to see that i) follows directly from (\ref{bl0986}).
	
Note that in this theorem, $u\in(0,1]$. By (\ref{bl0700}), the perturbation 
operator 
does not appear. So the proof of ii), iii), iv) here 
are totally the same as that of \cite[Theorem 3.10-3.12]{MR1214954} except 
for replacing the reference of \cite{MR1188532} there by the corresponding 
reference of \cite{MR1316553}.

Remark that the setting of this paper uses the language of Clifford modules, 
not the spin case in the references. However, there is no additional 
difficulty for this differences. The reason is that in each proof of Theorem 
\ref{e03022} ii), iii), iv), we localize the problem first. Locally, all 
manifolds are spin. 

The proof of Theorem \ref{e03022} is completed.
\end{proof}

Now we use the estimates in Theorem \ref{bl0584}, \ref{bl0588} and 
\ref{e03022} 
to prove Theorem \ref{bl0981}.

\begin{proof}[Proof of Theorem \ref{bl0981}]

From (\ref{bl0581}), we know that
\begin{multline}\label{bl0597}
\int_{\var}^{A}\beta_g^u(T_1,u)du-\int_{0}^{T_1}\beta_g^T(T,A)dT-\int_{\var}^{A}\beta_g^u(0,u)du+
\int_{0}^{T_1}\beta_g^T(T,\var)dT
\\
=I_1+I_2+I_3+I_4=0.
\end{multline}
We take the limits $A\rightarrow+\infty$, $T_1\rightarrow+\infty$ and then 
$\var\rightarrow 0$
in the indicated order. Let $I_j^k$, $j=1,2,3,4$, $k=1,2,3$ denote the value 
of the part $I_j$ after
the $k$th limit.

From Theorem \ref{bl0584}, (\ref{bl0583}) and the dominated convergence 
theorem, we conclude that
\begin{align}\label{bl0601}
I_1^3
=-\tilde{\eta}_g(\mF_Y, \mathcal{A}_Y).
\end{align}
Furthermore, by Theorem \ref{bl0588}, we get
\begin{align}\label{bl0600}
I_2^3=-\ch_g(\mathrm{sf}_G\{D(\mF_X)+\mA_X, D(\mF_{X}) 
+T_0\mathcal{V}+\mA_{T_0,Y}\}).
\end{align}
From (\ref{bl0575}), we obtain that
\begin{align}\label{bl0599}
I_3^3=\tilde{\eta}_g(\mF_X,\mA_X).
\end{align}

Finally, we calculate the last part.
By definition,
\begin{align}\label{bl0602}
I_4^0=\int_{0}^{T_1}\beta_g^T(T,\var)dT.
\end{align}
As $A\rightarrow +\infty$, $I_4^0$ remains constant and equal to $I_4^1$.
As $T_1\rightarrow+\infty$, by Theorem \ref{e03022} i),
\begin{align}\label{bl0681}
I_4^2=\int_{0}^{+\infty}\beta_g^T(T,\var)dT=\int_{0}^{+\infty}\var^{-1}\beta_g^T(T/\var,\var)dT.
\end{align}
Set
\begin{align}\label{bl0682}
\begin{split}
&K_1=\int_{0}^{1}\var^{-1}\beta_g^T(T/\var,\var)dT,
\\
&K_2=\int_{\var}^{1}\var^{-2}\beta_g^T(T/\var^{2},\var)dT,
\\
&K_3=\int_{1}^{+\infty}\var^{-2}\beta_g^T(T/\var^{2},\var)dT.
\end{split}
\end{align}
Clearly,
\begin{align}\label{bl0683}
I_4^2=K_1+K_2+K_3.
\end{align}

To simplify the notation, we denote by 
\begin{multline}
	\mD(T):=\frac{1}{2\sqrt{\pi}\sqrt{-1}}
	\frac{2^{\ell_2/2}}{\mathrm{det}^{1/2}(1-g|_{N_{X^g/X}})}
	\\
	\cdot\psi_{X^g} 
	\tr^{\mE_X/\mS}\left[\sigma_{\ell_2}(g^{\mE_X})\mathcal{V}|_{X^g}
	\exp\left(-R_{T^2}^{\mE_X/\mS}|_{X^g}\right)\right].
\end{multline}
Then by Definition \ref{bl0554}, after changing the variable, we have
\begin{align}
	\gamma_g^X(\mF_Y,\mF_X)=\int_0^{\infty}\mD(T)dT.
\end{align}
As $\var\rightarrow 0$, by Theorem \ref{e03022} ii),
\begin{align}\label{bl0684}
K_1\rightarrow -\int_{X^g}
\widehat{\mathrm{A}}_g(TX,\nabla^{TX})\cdot
\int_0^1\mD(T)dT.
\end{align}

We write $K_2$ in the form
\begin{multline}\label{bl0685}
K_2=\int_{\var}^{1}\frac{T}{\var}\left\{\var^{-1}\beta_g^T(T/\var^{2},\var)
+\int_{X^g}
\widehat{\mathrm{A}}_g(TX,\nabla^{TX})
\mD(T/\var) \right\} \frac{dT}{T}
\\
-\int_{X^g}
\widehat{\mathrm{A}}_g(TX,\nabla^{TX})
\int_1^{\var^{-1}}\mD(T)dT.
\end{multline}
By Theorem \ref{e03022} ii), there exist $C>0$, $\gamma\in (0,1]$ such that 
for $0<\var\leq T\leq 1$
\begin{multline}\label{bl0701}
\left|\frac{T}{\var}\left\{\var^{-1}\beta_g^T(T/\var^{2},\var)
+\int_{X^g}
\widehat{\mathrm{A}}_g(TX,\nabla^{TX})
\mD(T/\var) \right\} \right|
\\
\leq C\left(\var\left(1+\frac{T}{\var}\right)\right)^{\gamma}\leq 
C(2T)^{\gamma}.
\end{multline}
Using Theorem \ref{e03022} iii), (\ref{bl0545}), (\ref{bl0701}) and the 
dominated 
convergence theorem, as $\var\rightarrow 0$,
\begin{align}\label{bl0686}
K_2\rightarrow -\int_{X^g}\widehat{\mathrm{A}}_g(TX,\nabla^{TX})\ 
\int_1^{+\infty}\mD(T)dT.
\end{align}

Using Theorem \ref{e03022} iii), iv) and the dominated convergence theorem, 
we see that as $\var\rightarrow 0$,
\begin{align}\label{bl0687}
K_3\rightarrow 0.
\end{align}

Combining Definition \ref{bl0554}, (\ref{bl0683}), (\ref{bl0684}), 
(\ref{bl0686}) and 
(\ref{bl0687}), we see that as $\var\rightarrow 0$,
\begin{align}\label{e03032}
I_4^3=-\int_{X^g}\widehat{\mathrm{A}}_g(TX,\nabla^{TX})\,\gamma_g^X(\mF_Y,\mF_X).
\end{align}

Thus (\ref{bl0982}) follows from (\ref{bl0597})-(\ref{bl0599}) and 
(\ref{e03032}).

The proof of Theorem \ref{bl0981} is completed.

\end{proof}

\subsection{Proof of Theorem \ref{bl0563}}

In this subsection, we use the functoriality of the equivariant eta forms 
Theorem \ref{d188} to 
reduce Theorem \ref{bl0563} to the case when the base manifold is a 
point\footnote{The author thanks Prof. Xiaonan Ma for pointing out this 
	simplification, which is related to a remark in \cite[Section 
	7.5]{Bismut1997}}.
Recall that we may assume that $B$ is compact.

\begin{lemma}\label{bl1004} 
	There exist a $\Z_2$-graded 
	self-adjoint
	$C(TB)$-module $(\mE_B,h^{\mE_B})$ and a positive integer $q\in 
	\mathbb{Z}_+$ 
	such 
	that 
	$$\widehat{\mathrm{A}}(TB,\nabla^{TB})\ch(\mE_B/\mS,\nabla^{\mE_B})-q$$
	is an exact form for any Euclidean connection $\nabla^{TB}$ and Clifford 
	connection $\nabla^{\mE_B}$.    
\end{lemma}
\begin{proof}
	Let $(\mE_0,h^{\mE_0})$ be a $\Z_2$-graded self-adjoint $C(TB)$-module. 
Let $\nabla^{\mE_0}$ be a Clifford connection on $(\mE_0,h^{\mE_0})$.
	Then since the $G$-action is trivial on $B$, from the definition of the 
	$\widehat{\mathrm{A}}$-genus and 
	(\ref{bl1010}), there exists $m\in\Z_+$ such that 
	\begin{align}\label{bl1}
	\widehat{\mathrm{A}}(TB,\nabla^{TB})\cdot
	\ch(\mE_0/\mS,\nabla^{\mE_0})=m+\alpha,
	\end{align}
	where $\alpha\in \Omega^{\mathrm{even}}(B)$ is a closed form and $\deg 
	\alpha\geq 
	2$.
	Since $\alpha$ is nilpotent,
	\begin{align}\label{bl0963} 
	\{\widehat{\mathrm{A}}(TB,\nabla^{TB})\ch(\mE_0/\mS,\nabla^{\mE_0})\}^{-1}
	=\sum_{k=0}^{\infty}(-1)^k\frac{\alpha^k}{m^{k+1}}
	\end{align}
	is a closed well-defined even differential form over $B$. From the 
	isomorphism
	\begin{align}
	\ch: K^0(B)\otimes \mathbb{R}\overset{\simeq}{\rightarrow} 
	H^{\mathrm{even}}(B,\mathbb{R}),
	\end{align} 
	there exist non-zero real number $q\in\mathbb{R}$ and  virtual complex 
	vector 
	bundle 
	$E=E_+-E_-$, such that 
	$q^{-1}\ch([E])=[\{\widehat{\mathrm{A}}(TB,\nabla^{TB})
	\ch(\mE_0/\mS,\nabla^{\mE_0})\}^{-1}]$. Let $\nabla^E$ be a connection 
	on $E$.
	Let $\mE_B=\mE_0\widehat{\otimes} E$ and 
	$\nabla^{\mE_B}=\nabla^{\mE_0}\otimes 1+1\otimes\nabla^{E}$. Then
	$\ch(\mE_B/\mS,\nabla^{\mE_B})=\ch(\mE_0/\mS,\nabla^{\mE_0})\ch(E,\nabla^E)$.
	So we have  
	$$[\widehat{\mathrm{A}}(TB,\nabla^{TB})\ch(\mE_B/\mS,\nabla^{\mE_B})]
	=q\in	 
	H^{\mathrm{even}}(B,\mathbb{R}).$$ From (\ref{bl1}), we have $q\in \Z_+$.
	
	The proof of Lemma \ref{bl1004} is completed.
\end{proof}

Let $(\mE_B,h^{\mE_B})$ be the $C(TB)$-module
 in Lemma \ref{bl1004}. Let $\nabla^{\mE_B}$ be a Clifford connection 
on $(\mE_B,h^{\mE_B})$.
Thus 
\begin{align}\label{bl0961} 
\mF_{V}=(V,\pi_V^*\mE_B\widehat{\otimes}\mE_X, 
\pi_V^*g^{TB}\oplus g^{TX},\pi_V^*h^{\mE_B}\otimes h^{\mE_X}, 
\pi_V^*\nabla^{\mE_B}\otimes 1+1\otimes \nabla^{\mE_X})
\end{align}
is an equivariant geometric 
family over a point in $B$.
Let 
\begin{align}\label{bl0962} 
\mF_{V,t}=(V,\pi_V^*\mE_B\widehat{\otimes}\mE_X, 
g_t^{TV},\pi_V^*h^{\mE_B}\otimes h^{\mE_X}, 
\nabla_t^{\mE_{V}})
\end{align}
be the rescaled equivariant geometric family over a point constructed in the 
same way as 
in (\ref{d073}).

\begin{lemma}\label{bl1051} 
	There exist $t_0>0$ and
	$T'\geq 1$, such that for any $t\geq t_0$ and $T\geq T'$, the 
	operator $D(\mF_{V,t})+tT\mathcal{V}+t\mA_{T,Y}$ is invertible .
\end{lemma}
\begin{proof} 
	Let $f_1,\cdots,f_l$ be a locally orthonormal basis of $TB$. 
	Let $f_p^H$ be the horizontal lift of $f_p$ on $T^HV$.
	Let 
	$e_1,\cdots,e_n$ be a locally orthonormal basis of $TX$. 
	Set 
	\begin{align}
	D_t^B=c(f_p)\nabla_{f_p}^{\cE_X,u}+\frac{1}{8t}\la [f_{p}^H,f_q^H], 
	e_i\ra 
	c(e_i)c(f_p)c(f_q).
	\end{align}
	By \cite[(5.6)]{MR3626562}, we have
	\begin{align}
	D(\mF_{V,t})+tT\mathcal{V}+t\mA_{T,Y}=t(D(\mF_X)+T\mathcal{V}+\mA_{T,Y})
	+D_t^B.
	\end{align}

	From Proposition \ref{bl0676}, since $B$ is compact, there exist $c>0$ 
	and 
	$T'>0$, such that for any $s\in 
	\Lambda(T^*B)\widehat{\otimes}\cE_X$, $T\geq T'$, 
	\begin{align}\label{bl0946} 
	\|(D(\mF_X)+T\mathcal{V}+\mA_{T,Y})s\|_0^2\geq c^2\|s\|_0^2.
	\end{align}
	
	
	Let 
	\begin{align}
	R_{t,T}:=t[D(\mF_X)+T\mathcal{V}+\mA_{T,Y},D_t^B ]+D_t^{B,2}.
	\end{align}
	We have
	\begin{align}
	(D(\mF_{V,t})+tT\mathcal{V}+t\mA_{T,Y})^2
	=t^2(D(\mF_X)+T\mathcal{V}+\mA_{T,Y})^2+R_{t,T}.
	\end{align}
	
	Let $|\cdot|_{T,1}$ be the norm defined in the same way as 
	\cite[Definition 
	9.13]{Bismut1997}. In particular, 
	\begin{align}
	\|s\|_0\leq |s|_{T,1}.
	\end{align}
	Note that the perturbation operator $\mA_{T,Y}$ is uniformly bounded 
	with 
	respect to $T\geq 1$. From the arguments in the proof of \cite[Theorem 
	9.14]{Bismut1997}, we could obtain that there exist $C_1, C_2, C_3>0$, 
	such 
	that for $T\geq 1$, $t\geq 1$, $s\in 
	\Lambda(T^*B)\widehat{\otimes}\cE_X$,
	\begin{align}\label{bl0945} 
	\begin{split}
	&\|(D(\mF_X)+T\mathcal{V}+\mA_{T,Y})s\|_0^2\geq 
	C_1|s|_{T,1}^2-C_2\|s\|_0^2,
	\\
	&|\la R_{t,T}s, s\ra_0|\leq C_3\,t\,\|s\|_0\cdot|s|_{T,1}.
	\end{split}
	\end{align} 
	
	Take $\alpha=c^2/(c^2+2C_2)$. By (\ref{bl0946})-(\ref{bl0945}), for 
	$T\geq 
	T'$, $t\geq 1$, we have
	\begin{multline}
	\|(D(\mF_{V,t})+tT\mathcal{V}+t\mA_{T,Y})s\|_0^2
	=|\la t^2(D(\mF_X)+T\mathcal{V}+\mA_{T,Y})^2s+ R_{t,T}s,s \ra_0|
	\\
	\geq (1-\alpha)t^2\|(D(\mF_X)+T\mathcal{V}+\mA_{T,Y})s\|_0^2+\alpha 
	t^2\|(D(\mF_X)+T\mathcal{V}+\mA_{T,Y})s\|_0^2-|\la R_{t,T}s, s\ra_0|
	\\
	\geq (1-\alpha)c^2t^2\|s\|_0^2+\alpha C_1t^2|s|_{T,1}^2-\alpha 
	C_2t^2\|s\|_{0}^2-C_3t\|s\|_0\cdot|s|_{T,1}
	\\
	\geq \alpha C_2t^2\|s\|_{0}^2+t(\alpha C_1t-C_3)|s|_{T,1}^2.
	\end{multline}
	Take $t_0=\max\{2C_3/\alpha C_1, 1 \}$. For any $t\geq t_0$, $T\geq 
	T'$, 
	there exists $C>0$, such that
	\begin{align}\label{bl0964} 
	\|(D(\mF_{V,t})+tT\mathcal{V}+t\mA_{T,Y})s\|_0^2\geq Ct^2\|s\|_{0}^2.
	\end{align}
	Since $D(\mF_{V,t})+tT\mathcal{V}+t\mA_{T,Y}$ is self-adjoint, by 
	(\ref{bl0964}), it is surjective. Thus 
	$D(\mF_{V,t})+tT\mathcal{V}+t\mA_{T,Y}$ is invertible.
	
	The proof of Lemma \ref{bl1051} is completed.
\end{proof}

Let
\begin{align}\label{bl0966} 
\mF_{W}=(W,\pi_W^*\mE_B\widehat{\otimes}\mE_Y, 
\pi_W^*g^{TB}\oplus g^{TY},\pi_W^*h^{\mE_B}\otimes h^{\mE_Y}, 
\pi_W^*\nabla^{\mE_B}\otimes 1+1\otimes \nabla^{\mE_Y})
\end{align}
be the equivariant geometric 
family over a point in $B$.
Let 
\begin{align}\label{bl0965} 
\mF_{W,t}=(W,\pi_W^*\mE_B\widehat{\otimes}\mE_Y, 
g_t^{TW},\pi_V^*h^{\mE_B}\otimes h^{\mE_Y}, 
\nabla_t^{\mE_{W}})
\end{align}
be the rescaled equivariant geometric family constructed in the same way as 
in (\ref{d073}) and (\ref{bl0962}).

Firstly, we assume that $B$ is closed.
Let $t_0$ be the constant taking in Lemma \ref{bl1051}. We may assume that 
when $t\geq t_0$, $D(\mF_{W,t})+1\widehat{\otimes}t\mA_Y$ is invertible by 
the arguments 
before Theorem \ref{d188}.
By Theorem \ref{d188} and Lemma \ref{bl1051}, we have
\begin{multline}\label{bl0949} 
\widetilde{\eta}_g(\mF_{W, t_0}, 
1\widehat{\otimes}t_0\mA_{Y})=\int_{B}\widehat{\mathrm{A}}(TB,\nabla^{TB})
\ch(\mE_B/\mS,\nabla^{\mE_B})\, \widetilde{\eta}_g(\mF_{Y},\mA_Y)
\\
-\int_{W^g}\left(\widetilde{\widehat{\mathrm{A}}}_g\cdot
\widetilde{\ch_g}\right)\left(\nabla_{t_0}^{TW}, \nabla^{TB,TY}, 
\nabla_{t_0}^{\mE_W}, \,^0\nabla^{\mE_W}\right)
\end{multline}
and
\begin{multline}
\widetilde{\eta}_g(\mF_{V, t_0}, 1\widehat{\otimes}(t_0T'\mathcal{V}+ 
t_0\mA_{T',Y}))
\\
=\int_{B}\widehat{\mathrm{A}}(TB,\nabla^{TB})
\ch(\mE_B/\mS,\nabla^{\mE_B})\, 
\widetilde{\eta}_g(\mF_{X},T'\mathcal{V}+\mA_{T',Y})
\\
-\int_{V^g}\left(\widetilde{\widehat{\mathrm{A}}}_g\cdot
\widetilde{\ch_g}\right)\left(\nabla_{t_0}^{TV}, \nabla^{TB,TX}, 
\nabla_{t_0}^{\mE_V}, \,^0\nabla^{\mE_V}\right).
\end{multline}

Set 
\begin{multline}
\Delta_B=\tilde{\eta}_g(\mF_X, 
T'\mathcal{V}+A_{T',Y})-\tilde{\eta}_g(\mF_Y, 
\mathcal{A}_Y)
\\
-\int_{X^g}\widehat{\mathrm{A}}_g(TX, 
\nabla^{TX})\,\gamma_g^X(\mF_Y,\mF_X)\in \Omega^*(B,\C)/\Im d.
\end{multline}
From \cite[(1.17)]{MR2273508}, Theorem \ref{bl0558} and (\ref{e01085}), , we 
have $d^B\Delta_B=0$.

Recall that $\mathcal{V}\in \End_{C(TX)}^{}(\mE_X)$ satisfies the 
fundamental 
assumption (\ref{bl1011}) with respect to $\mF_Y$ and $\mF_X$. 
Let $1\widehat{\otimes} \mathcal{V}$ is the extension of $\mV$ on 
$\pi_V^*\mE_B\widehat{\otimes}\mE_X$ in the same way as $\mA$ in 
(\ref{bl0995}).
Then 
$1\widehat{\otimes} \mathcal{V}$ satisfies the 
fundamental assumption (\ref{bl1011}) with respect to $\mF_W$ and $\mF_V$. 
Furthermore, $1\widehat{\otimes} t\mathcal{V}$ satisfies the fundamental 
assumption 
(\ref{bl1011}) with respect to $\mF_{W,t}$ and $\mF_{V,t}$. Observe that 
$\gamma_g^V(\mF_{W,t},\mF_{V,t})$ does not depend on $t$. We also denote it 
by $\gamma_g^V(\mF_{W},\mF_{V})$.

From Theorem \ref{bl0981}, if $\dim V$ is odd, there exists $T_0>0$ such that
\begin{multline}\label{bl0954}
\tilde{\eta}_g(\mF_{V,t_0}, \mA_{V})=\tilde{\eta}_g(\mF_{W,t_0}, 
1\widehat{\otimes}t_0\mathcal{A}_{Y})+\int_{V^g}\widehat{\mathrm{A}}_g(TX, 
\nabla_{t_0}^{TX})\,\gamma_g^X(\mF_{W},\mF_{V})
\\
+\ch_g(\mathrm{sf}_G\{D(\mF_{V,t_0})+\mA_V, D(\mF_{V,t_0}) 
+1\widehat{\otimes} (T_0t_0\mathcal{V}+ t_0\mA_{T_0,Y})\})
.
\end{multline}
We may assume that $T_0\geq T'$, which is determined in Lemma \ref{bl1051}.
By anomaly formula Theorem \ref{e01105}, we have
\begin{multline}\label{bl0947} 
\tilde{\eta}_g(\mF_{V,t_0},1\widehat{\otimes} (T_0t_0\mathcal{V}+ 
t_0\mA_{T_0,Y}))=\tilde{\eta}_g(\mF_{W,t_0}, 
1\widehat{\otimes}t_0\mathcal{A}_{Y})
\\
+\int_{V^g}\widehat{\mathrm{A}}_g(TV, 
\nabla_{t_0}^{TV})\,\gamma_g^V(\mF_{W},\mF_{V}).
\end{multline}
Note that if $\dim V$ is even, (\ref{bl0947}) also holds, because in this 
case all terms 
in (\ref{bl0947}) vanish.

From the anomaly formula Theorem \ref{e01105}, Lemma \ref{bl1051} and 
(\ref{bl0947}), for 
$t>t_0$, 
we have
\begin{multline}\label{bl0948} 
\int_{V^g}\left(\widetilde{\widehat{\mathrm{A}}}_g\cdot
\widetilde{\ch_g}\right)\left(\nabla_{t_0}^{TV}, \nabla_{t}^{TV}, 
\nabla_{t_0}^{\mE_V}, \nabla_t^{\mE_V}\right)
\\
=\int_{W^g}\left(\widetilde{\widehat{\mathrm{A}}}_g\cdot
\widetilde{\ch_g}\right)\left(\nabla_{t_0}^{TW}, \nabla_{t}^{TW}, 
\nabla_{t_0}^{\mE_W}, \nabla_t^{\mE_W}\right)
\\
-\int_{V^g}\widehat{\mathrm{A}}_g(TV, 
\nabla_{t_0}^{TV})\,\gamma_g^V(\mF_{W},\mF_{V})
\\
+\int_{V^g}\widehat{\mathrm{A}}_g(TV, 
\nabla_{t}^{TV})\,\gamma_g^V(\mF_{W},\mF_{V}).
\end{multline}
Note that locally the manifolds are spin.
From \cite[Proposition 4.5]{MR3626562} and the arguments in \cite[Section 
5.5]{MR3626562}, we have
\begin{multline}
\lim_{t\rightarrow +\infty}\left(\widetilde{\widehat{\mathrm{A}}}_g\cdot
\widetilde{\ch_g}\right)\left(\nabla_{t_0}^{TV}, \nabla_{t}^{TV}, 
\nabla_{t_0}^{\mE_V}, \nabla_t^{\mE_V}\right)
\\
=\left(\widetilde{\widehat{\mathrm{A}}}_g\cdot
\widetilde{\ch_g}\right)\left(\nabla_{t_0}^{TV}, \nabla^{TB,TX}, 
\nabla_{t_0}^{\mE_V}, \,^0\nabla^{\mE_V}\right)
\end{multline}
and
\begin{multline}
\lim_{t\rightarrow +\infty}\widehat{\mathrm{A}}_g(TV, 
\nabla_{t}^{TV})=\widehat{\mathrm{A}}_g(TV, 
\nabla^{TB,TX})
\\
=\pi_V^*\widehat{\mathrm{A}}(TB, 
\nabla^{TB})\cdot\widehat{\mathrm{A}}_g(TX, \nabla^{TX}).
\end{multline}

By Definition \ref{bl0554}, we have
\begin{align}\label{bl0950}
\gamma_g^V(\mF_{W},\mF_{V})=\ch(\mE_B/\mS,\nabla^{\mE_B})\gamma_g^X(\mF_{Y},\mF_{X}).
\end{align}

From Lemma \ref{bl1004} and (\ref{bl0949})-(\ref{bl0950}), if $B$ is 
closed, we have
\begin{align}\label{bl0953} 
\int_B\Delta_B=q^{-1}\cdot\int_B 
\widehat{\mathrm{A}}(TB,\nabla^{TB})\ch(\mE_B/\mS,\nabla^{\mE_B})\cdot\Delta_B=0.
\end{align}

In general, $B$ is not necessary closed. Let $K$ be a closed submanifold 
of $B$. 
Let $\mF_Y|_K$ and $\mF_X|_K$ be the restrictions of $\mF_Y$ and $\mF_X$ on 
$K$. Let $T_0\geq 1$ be the constant determined in Proposition \ref{bl0676} 
associated with $B$. Then $(T\mV+\mA_{T,Y})|_K$ is the perturbation operator 
with respect to $D(\mF_X|_K)$ over $K$ for $T\geq T_0$. Set
\begin{multline}
\Delta_K=\tilde{\eta}_g(\mF_X|_K, 
(T_0\mathcal{V}+A_{T_0,Y})|_K)-\tilde{\eta}_g(\mF_Y|_K, 
\mathcal{A}_Y|_K)
\\
+\int_{X^g}\widehat{\mathrm{A}}_g(TX, 
\nabla^{TX})\,\gamma_g^X(\mF_Y|_K,\mF_X|_K)\in \Omega^*(K,\C)/\Im d.
\end{multline}
From Definition \ref{d017} and \ref{bl0554}, we could see that
$
\int_K\Delta_B=\int_K\Delta_K.
$

On the other hand, from (\ref{bl0953}), we have
$
\int_K\Delta_K=0.
$
So for any closed submanifold $K$ of $B$, we have 
\begin{align}\label{bl0998} 
\int_K\Delta_B=0.
\end{align}
From the arguments in \cite[\S 21, \S 22]{MR0346830}, we obtain that 
$\Delta_B$ is 
exact on $B$. Therefore, we obtain Theorem \ref{bl0563} from the anomaly 
formula 
Theorem \ref{e01105}.

The proof of our main result is completed.

\vspace{3mm}\textbf{Acknowledgements}
I would like to thank Prof. Xiaonan Ma for many discussions. I would like to 
thank the Mathematical Institute of
Humboldt-Universit\"at zu Berlin and University of California, Santa 
Barbara, especially Prof. J. Br\"uning and Xianzhe Dai for financial support 
and hospitality.  Part of the work was done while I was visiting Institut 
des Hautes \'Etudes Scientifiques (IHES) and Max Planck Institute for 
Mathematics (MPIM) which I thank the financial support. 

This research has been financially supported by  the China Postdoctoral 
Science Foundation
(No.2017M621404).


\begin{thebibliography}{10}
	
	\bibitem[AH59]{MR0110106}
	M.~F. Atiyah and F.~Hirzebruch.
	\newblock Riemann-{R}och theorems for differentiable manifolds.
	\newblock {\em Bull. Amer. Math. Soc.}, 65:276--281, 1959.
	
	\bibitem[APS75]{MR0397797}
	M.~F. Atiyah, V.~K. Patodi, and I.~M. Singer.
	\newblock Spectral asymmetry and {R}iemannian geometry. {I}.
	\newblock {\em Math. Proc. Cambridge Philos. Soc.}, 77:43--69, 1975.
	
	\bibitem[AS69]{MR0285033}
	M.~F. Atiyah and I.~M. Singer.
	\newblock Index theory for skew-adjoint {F}redholm operators.
	\newblock {\em Inst. Hautes \'Etudes Sci. Publ. Math.}, 37:5--26, 1969.
	
	\bibitem[ASIV]{MR0279833}
	M.~F. Atiyah and I.~M. Singer.
	\newblock The index of elliptic operators. {IV}.
	\newblock {\em Ann. of Math. (2)}, 93:119--138, 1971.
	
	\bibitem[BGV]{MR2273508}
	N.~Berline, E.~Getzler, and M.~Vergne.
	\newblock {\em Heat kernels and {D}irac operators}.
	\newblock Grundlehren Text Editions. Springer-Verlag, Berlin, 2004.
	\newblock Corrected reprint of the 1992 original.
	
	\bibitem[B86]{Bismut1985}
	J.-M. Bismut.
	\newblock The {A}tiyah-{S}inger index theorem for families of {D}irac
	operators: two heat equation proofs.
	\newblock {\em Invent. Math.}, 83(1):91--151, 1986.
	
	\bibitem[B95]{MR1316553}
	J.-M. Bismut.
	\newblock Equivariant immersions and {Q}uillen metrics.
	\newblock {\em J. Differential Geom.}, 41(1):53--157, 1995.
	
	\bibitem[B97]{Bismut1997}
	J.-M. Bismut.
	\newblock Holomorphic families of immersions and higher analytic torsion 
	forms.
	\newblock {\em Ast\'erisque}, (244):viii+275, 1997.
	
	\bibitem[BC89]{MR966608}
	J.-M. Bismut and J.~Cheeger.
	\newblock {$\eta$}-invariants and their adiabatic limits.
	\newblock {\em J. Amer. Math. Soc.}, 2(1):33--70, 1989.
	
	\bibitem[BF86]{MR853982}
	J.-M. Bismut and D.~Freed.
	\newblock The analysis of elliptic families. {I}. {M}etrics and 
	connections on
	determinant bundles.
	\newblock {\em Comm. Math. Phys.}, 106(1):159--176, 1986.
	
	\bibitem[BL92]{MR1188532}
	J.-M. Bismut and G.~Lebeau.
	\newblock Complex immersions and {Q}uillen metrics.
	\newblock {\em Inst. Hautes \'Etudes Sci. Publ. Math.}, 74:ii+298 pp. 
	1992.
	
	\bibitem[BM04]{MR2097553}
	J.-M. Bismut and X.~Ma.
	\newblock Holomorphic immersions and equivariant torsion forms.
	\newblock {\em J. Reine Angew. Math.}, 575:189--235, 2004.
	
	\bibitem[BZ93]{MR1214954}
	J.-M. Bismut and W.~Zhang.
	\newblock Real embeddings and eta invariants.
	\newblock {\em Math. Ann.}, 295(4):661--684, 1993.
	
	\bibitem[BuM04]{MR2072502}
	U.~Bunke and X.~Ma.
	\newblock Index and secondary index theory for flat bundles with duality.
	\newblock In {\em Aspects of boundary problems in analysis and geometry},
	volume 151 of {\em Oper. Theory Adv. Appl.}, pages 265--341. 
	Birkh\"auser,
	Basel, 2004.
	
	\bibitem[BuS09]{Bunke2009}
	U.~Bunke and T.~Schick.
	\newblock Smooth {$K$}-theory.
	\newblock {\em Ast\'erisque}, 328:45--135 (2010), 2009.
	
	\bibitem[BuS13]{MR3182258}
	U.~Bunke and T.~Schick.
	\newblock Differential orbifold {K}-theory.
	\newblock {\em J. Noncommut. Geom.}, 7(4):1027--1104, 2013.
	
	\bibitem[D91]{MR1088332}
	X.~Dai.
	\newblock Adiabatic limits, nonmultiplicativity of signature, and {L}eray
	spectral sequence.
	\newblock {\em J. Amer. Math. Soc.}, 4(2):265--321, 1991.
	
	\bibitem[DZ98]{MR1638328}
	X.~Dai and W.~Zhang.
	\newblock Higher spectral flow.
	\newblock {\em J. Funct. Anal.}, 157(2):432--469, 1998.
	
	\bibitem[DZ00]{MR1870658}
	X.~Dai and W.~Zhang.
	\newblock Real embeddings and the {A}tiyah-{P}atodi-{S}inger index 
	theorem for
	{D}irac operators.
	\newblock {\em Asian J. Math.}, 4(4):775--794, 2000.
	\newblock Loo-Keng Hua: a great mathematician of the twentieth century.
	
	\bibitem[DR73]{MR0346830}
	G.~de~Rham.
	\newblock {\em Vari\'et\'es diff\'erentiables. {F}ormes, courants, formes
		harmoniques}.
	\newblock Hermann, Paris, 1973.
	\newblock Troisi{\`e}me {\'e}dition revue et augment{\'e}e, Publications 
	de
	l'Institut de Math{\'e}matique de l'Universit{\'e} de Nancago, III,
	Actualit{\'e}s Scientifiques et Industrielles, No. 1222b.
	
	\bibitem[D78]{MR511246}
	H.~Donnelly.
	\newblock Eta invariants for {$G$}-spaces.
	\newblock {\em Indiana Univ. Math. J.}, 27(6):889--918, 1978.
	
	\bibitem[FXZ09]{MR2532714}
	H.~Feng, G.~Xu, and W.~Zhang.
	\newblock Real embeddings, {$\eta$}-invariant and {C}hern-{S}imons 
	current.
	\newblock {\em Pure Appl. Math. Q.}, 5(3, Special Issue: In honor of 
	Friedrich
	Hirzebruch. Part 2):1113--1137, 2009.
	
	\bibitem[FL10]{MR2602854}
	D.~S. Freed and J.~Lott.
	\newblock An index theorem in differential {$K$}-theory.
	\newblock {\em Geom. Topol.}, 14(2):903--966, 2010.
	
	\bibitem[LM89]{MR1031992}
	H.~B. Lawson and M.-L. Michelsohn.
	\newblock {\em Spin geometry}, volume~38 of {\em Princeton Mathematical
		Series}.
	\newblock Princeton University Press, Princeton, NJ, 1989.
	
	\bibitem[Liu16]{Liu2016}
	B.~Liu.
	\newblock Equivariant eta forms and equivariant differential 
	{$K$}-theory.
	\newblock {\em arXiv: 1610.02311}, 2016.
	
	\bibitem[Liu17]{MR3626562}
	B.~Liu.
	\newblock Functoriality of equivariant eta forms.
	\newblock {\em J. Noncommut. Geom.}, 11(1):225--307, 2017.
	
	\bibitem[M02]{MR1942300}
	X.~Ma.
	\newblock Functoriality of real analytic torsion forms.
	\newblock {\em Israel J. Math.}, 131:1--50, 2002.
	
	\bibitem[MM07]{MR2339952}
	X.~Ma and G.~Marinescu.
	\newblock {\em Holomorphic {M}orse inequalities and {B}ergman kernels}, 
	volume
	254 of {\em Progress in Mathematics}.
	\newblock Birkh\"auser Verlag, Basel, 2007.
	
%
	
	\bibitem[Q85]{MR790678}
	D.~Quillen.
	\newblock Superconnections and the {C}hern character.
	\newblock {\em Topology}, 24(1):89--95, 1985.
	
	\bibitem[S68]{MR0234452}
	G.~Segal.
	\newblock Equivariant {$K$}-theory.
	\newblock {\em Inst. Hautes \'Etudes Sci. Publ. Math.}, 34:129--151, 
	1968.
	
	\bibitem[Z01]{MR1864735}
	W. Zhang.
	\newblock {\em Lectures on {C}hern-{W}eil theory and {W}itten 
	deformations},
	volume~4 of {\em Nankai Tracts in Mathematics}.
	\newblock World Scientific Publishing Co., Inc., River Edge, NJ, 2001.
	
	\bibitem[Z05]{MR2129895}
	W.~Zhang.
	\newblock {$\eta$}-invariant and {C}hern-{S}imons current.
	\newblock {\em Chinese Ann. Math. Ser. B}, 26(1):45--56, 2005.
	

	
\end{thebibliography}



\end{document}